\documentclass[11pt]{amsart}
\usepackage{amssymb,amsxtra,amsmath,amsfonts}
\usepackage{eucal,mathrsfs}
\usepackage{verbatim,enumerate}
\usepackage{graphicx,color,graphics} 
\usepackage{color}
\usepackage[linktocpage=true,colorlinks=true, linkcolor=blue, citecolor=red, urlcolor=green]{hyperref}

\newtheorem{theorem}{Theorem}
\newtheorem{lemma}[theorem]{Lemma}
\newtheorem{proposition}[theorem]{Proposition}
\newtheorem{corollary}[theorem]{Corollary}

\theoremstyle{definition}
\newtheorem{definition}[theorem]{Definition}
\newtheorem{example}[theorem]{Example}

\theoremstyle{remark}
\newtheorem{remark}[theorem]{Remark}

\numberwithin{equation}{section}
\numberwithin{theorem}{section}

\newcommand\thref{Theorem \ref}
\newcommand\leref{Lemma \ref}
\newcommand\prref{Proposition \ref}
\newcommand\coref{Corollary \ref}
\newcommand\deref{Definition \ref}
\newcommand\exref{Example \ref}
\newcommand\reref{Remark \ref}
\newcommand\seref{Section \ref}
\renewcommand{\comment}[1]{}
\def\CC{\mathbb{C}}

\def\cR{\mathcal{R}}

\def\S{\mathscr{S}}
\def\N{\mathscr{S}}

\def\V{\mathcal{V}}

\def\A{\mathcal{A}}

\def\C{\mathbb{C}}

\def\S{\mathscr{S}}
\def\T{\mathcal{T}}
\def\U{\mathcal{U}}
\def\ZZ{\mathbb{Z}}

\DeclareMathOperator\End{End}

\DeclareMathOperator\Hom{Hom}

\DeclareMathOperator\ad{ad}

\DeclareMathOperator\gr{gr}

\DeclareMathOperator\Span{span}

\DeclareMathOperator\LF{LFie}

\DeclareMathOperator\F{F}

\def\vac{{\boldsymbol{1}}}  

\def\al{\alpha}
\def\be{\beta}

\def\la{\lambda}

\def\ze{\zeta}
\def\z{z}
\def\d{\partial}

\def\lieg{{\mathfrak{g}}}

\begin{document}

\title[Log Vertex Algebras and Non-local Poisson Vertex Algebras]{Logarithmic Vertex Algebras and Non-local Poisson Vertex Algebras}

\author{Bojko Bakalov}
\address{Department of Mathematics,
North Carolina State University,
Raleigh, NC 27695, United States}
\email{bojko\_bakalov@ncsu.edu}

\author{Juan J. Villarreal}
\address{Department of Mathematical Sciences, 
University of Bath, 
Bath BA2 7AY,
United Kingdom}
\email{juanjos3villarreal@gmail.com}


\dedicatory{Dedicated to Victor G.\ Kac on his 80th birthday}

\date{October 27, 2022; Revised August 20, 2023}



\begin{abstract}
Logarithmic vertex algebras were introduced in our previous paper, motivated by logarithmic conformal field theory \cite{BV}. Non-local Poisson vertex algebras were introduced by De Sole and Kac, motivated by the theory of integrable systems \cite{DK}. We prove that the associated graded vector space of any filtered logarithmic vertex algebra has an induced structure of a non-local Poisson vertex algebra. We use this relation to obtain new examples of both logarithmic vertex algebras and non-local Poisson vertex algebras.
\end{abstract}

\maketitle

\tableofcontents

\section{Introduction}\label{sec1}

We start by reviewing the well-known fact that the semi-classical limit of an associative algebra is a Poisson algebra.
Suppose that $A$ is an associative algebra over $\CC$ with an increasing \emph{filtration} by subspaces
\[
\{0\}=\cdots=\F^{-1}\! A\subset \F^{0}\! A\subset \F^{1}\! A\subset \F^{2}\! A\subset\cdots\subset A \,, 
\]
which is compatible with the product in $A$, so that
\[ 
ab \in\F^{m+k}\! A \quad\text{ for }\quad a\in\F^{m}\! A \,, \;\; b\in\F^{k}\! A \,.
\]
Then the \emph{associated graded} $\gr A$, defined by 
\[
\gr A:=\bigoplus_{k=0}^\infty \gr^{k}\! A\, , \qquad \gr^{k}\! A:=\F^{k}\! A/\F^{k-1}\! A\,, 
\]
inherits from $A$ an associative product of degree $0$. Explicitly, the product in $\gr A$ is given by
\[ 
\bar{a}\bar{b} := ab + \F^{m+k-1}\! A \in \gr^{m+k}\! A
\]
for
\[ 
\bar{a} = a+\F^{m-1}\! A \in \gr^{m}\! A \,, \qquad
\bar{b} = b + \F^{k-1}\! A \in \gr^{k}\! A \,,
\]
and it does not depend on the choice of representatives $a\in\F^{m}\! A$ and $b\in\F^{k}\! A$.
Assume now that the product in $\gr A$ is commutative, i.e.,
\[ 
ab-ba \in\F^{m+k-1}\! A \quad\text{ for }\quad a\in\F^{m}\! A \,, \;\; b\in\F^{k}\! A \,.
\]
Then we can also define a Lie bracket of degree $-1$ on $\gr A$ by
\[ 
[\bar{a},\bar{b}] := ab-ba +\F^{m+k-2}\! A \in\gr^{m+k-1}\! A \,,
\]
which together with the commutative associative product makes $\gr A$ a \emph{Poisson algebra}.
This situation happens, for example, when $A=U(L)$ is the universal enveloping algebra of a Lie algebra $L$.
Then $A$ has a canonical filtration such that $\F^{0}\! A = \CC$, $\F^{1}\! A=L\oplus\CC$, and the above properties hold.
The corresponding Poisson algebra $\gr A$ is isomorphic to the symmetric algebra $S(L)$,
by the Poincar\'e--Birkhoff--Witt Theorem.

The above well-known results can be generalized to vertex algebras. 
Recall first that a \emph{vertex algebra} \cite{Bo} is a vector
space $V$ equipped with infinitely many bilinear products
\begin{equation}\label{e1.1}
\mu_{(n)} \colon V \otimes V \to V \,, \qquad n\in\ZZ\,,
\end{equation}
subject to certain axioms, the main one being the Borcherds identity (see \cite{FB,FLM,K1,LL}).
It is customary to denote the $n$-th product of $a,b\in V$ as $a_{(n)}b = \mu_{(n)}(a\otimes b)$,
and one of the axioms states that for every fixed $a,b$ we have $a_{(n)}b =0$ for sufficiently large $n$.
A special case of the Borcherds identity is the commutator formula ($a,b,c\in V$, $m,k\in\ZZ$):
\begin{equation}\label{e1.2}
a_{(m)} (b_{(k)} c) - b_{(k)} (a_{(m)}c) = \sum_{j=0}^\infty \binom{m}{j} (a_{(j)}b)_{(m+k-j)} c \,,
\end{equation}
with the sum over $j$ in fact finite.
The commutator formula implies in particular that the linear operators $a_{(m)} \in\End(V)$ span a Lie algebra.
Important examples of vertex algebras can be constructed from representations of infinite-dimensional Lie algebras
such as affine Kac--Moody algebras and the Virasoro algebra (see \cite{K1,KRR}).
Vertex algebras provide a rigorous algebraic approach to $2$-dimensional conformal field theory (see e.g.\ \cite{DMS}).

If we restrict to $m,k\ge 0$ in \eqref{e1.2}, we get an identity satisfied by the products $\mu_{(n)}$ for $n\ge0$. 
It is convenient to organize these products in a generating function
\begin{equation}\label{e1.3}
[a_\la b] := \sum_{n=0}^\infty \frac{\la^n}{n!} \, \mu_{(n)}(a\otimes b) \,,
\end{equation}
called the \emph{$\la$-bracket}, where $\la$ is a formal variable. Note that $[a_\la b] \in V[\la]$ is a polynomial in $\la$.
Then the collection of identities \eqref{e1.2} for all $m,k\ge 0$ is equivalent to the Jacobi identity
\[
[a_\la [b_\mu c]] - [b_\mu [a_\la c]] = [[a_\la b]_{\la+\mu} c] 
\]
for formal variables $\la,\mu$. Together with a few other axioms, this gives rise to the notion of a \emph{Lie conformal algebra}
\cite{K1} as a vector space $R$ equipped with a $\la$-bracket $R \otimes R \to R[\la]$.
From a Lie conformal algebra $R$, one can construct a universal vertex algebra $V(R)$, which is reminiscent of the universal
enveloping algebra of a Lie algebra (see \cite{K1,BK}).

A \emph{Poisson vertex algebra} \cite{FB,Li} is defined as a Lie conformal algebra that also has a commutative associative product
satisfying the Leibniz rule
\[
[a_\la (bc)] = [a_\la b]c + b[a_\la c] \,.
\]
Poisson vertex algebras have important applications to integrable systems; see \cite{K2} for an introduction. 
Examples of Poisson vertex algebras can be constructed as the symmetric algebras $S(R)$ of Lie conformal algebras $R$.

A vertex algebra $V$ is called \emph{filtered} \cite{Li} if it has an increasing filtration by subspaces
\[
\{0\}=\cdots=\F^{-1}V\subset \F^{0}V\subset \F^{1}V\subset \F^{2}V\subset\cdots\subset V \,, 
\]
such that
\begin{equation}\label{e1.4}
\mu_{(n)}(\F^{m}V\otimes \F^{k}V)\subset 
                \begin{cases}
                  \F^{m+k}V\, ,  \; & n<0 \,, \\
                  \F^{m+k-1}V\, , \;  &  n\geq 0\, .
                \end{cases}
\end{equation}
It was proved in \cite{Li} (and in a different form in \cite{FB}) that the associated graded $\gr V$
has a natural structure of a Poisson vertex algebra defined as follows.
The commutative associative product in $\gr V$ is induced from the product $\mu_{(-1)}$:
\begin{equation}\label{e1.5}
\begin{split}
&\bar{a}\bar{b} := \mu_{(-1)}(a \otimes b) + \F^{m+k-1}V \in \gr^{m+k} V \,, \\
&\quad\text{for}\quad \bar{a} = a+\F^{m-1}V \in \gr^{m} V \,, \quad
\bar{b} = b + \F^{k-1}V \in \gr^{k} V \,,
\end{split}
\end{equation}
while the
$\la$-bracket in $\gr V$ is induced from the $\la$-bracket \eqref{e1.3}:
\[ 
[\bar{a}_\la\bar{b}] := [a_\la b] + \F^{m+k-2}V \in (\gr^{m+k-1} V)[\la] \,.
\]
In particular, for the universal vertex algebra $V(R)$ of a Lie conformal algebra $R$,
one has $\gr V(R) \cong S(R)$ (see e.g.\ \cite[Section 2.5]{BDSK}).

Recently, we introduced a generalization of the notion of a vertex algebra, called a \emph{logarithmic vertex algebra} \cite{BV}.
The main motivation behind this notion is logarithmic conformal field theory \cite{G1, G2, CR}, and logarithmic vertex algebras provide a rigorous algebraic formulation of quantum fields with logarithmic singularities in their operator product expansion. Logarithmic vertex algebras are also motivated by twisted logarithmic modules \cite{B,H} (see also \cite{BS1, BS2}). 
Besides bilinear products $\mu_{(n)}$ as in \eqref{e1.1}, a logarithmic vertex algebra $V$ is endowed with a linear map
\[
\N\colon V \otimes V \to V \otimes V \,,
\]
called the \emph{braiding map}; see \cite{BV} and Section \ref{sec2.3} below.
The main axiom satisfied by the products $\mu_{(n)}$ is the Borcherds identity \eqref{borcherds2},
which involves the braiding map $\N$ and reduces to the Borcherds identity for ordinary vertex algebras when $\N=0$.
In particular, a logarithmic vertex algebra with a zero braiding map is the same as an ordinary vertex algebra.

In \cite{DK}, the notion of a Lie conformal algebra was generalized to that of a \emph{non-local Lie conformal algebra} $R$,
in which the $\la$-bracket is no longer a polynomial in $\la$ but is a Laurent series in $\la^{-1}$.
Denoting this more general $\la$-bracket as $\{a_\la b\}$, we have
\[
\{a_\la b\} \in R(\!(\lambda^{-1})\!) := R[[\lambda^{-1}]][\la] \,, \qquad a,b\in R\,.
\]
There are subtleties in defining the compositions of brackets in the Jacobi identity
\[
\{a_\la \{b_\mu c\}\} - \{b_\mu \{a_\la c\}\} = \{\{a_\la b\}_{\la+\mu} c\} \,,
\]
which are discussed in \cite{DK} and in Section \ref{sec2.2} below.
Then a \emph{non-local Poisson vertex algebra} \cite{DK} is defined as a non-local Lie conformal algebra with an additional commutative associative product
satisfying the Leibniz rule
\[
\{a_\la (bc)\} = \{a_\la b\}c + b\{a_\la c\} \,.
\]
Non-local Poisson vertex algebras have applications to integrable systems, and several interesting examples were presented in \cite{DK}; see also \cite{DKV,DKV2, DKVW} for further developments.

In this paper, we introduce the notion of a \emph{filtered} logarithmic vertex algebra $V$, by requiring \eqref{e1.4} and
\begin{equation*}
\N(\F^{m} V\otimes \F^{k} V)\subset \F^{m+k-1} (V\otimes V) := \bigoplus_{j=0}^{m+k-1}\F^{j} V\otimes \F^{m+k-1-j}V \,.
\end{equation*}
Then we define a $\la$-bracket on $V$ by
\begin{equation}\label{e1.7}
\{a_{\lambda}b\}:=[a_{\lambda}b]
+\sum_{n=0}^\infty (-1)^{n}n!\, \lambda^{-n-1} \, \mu_{(-n-1)}\bigl(\N(a\otimes b)\bigr),
\end{equation}
where $[a_{\lambda}b]$ is given as before by \eqref{e1.3}.
The associated graded $\gr V$ is equipped with the product defined by \eqref{e1.5} and by the induced $\la$-bracket
\begin{equation}\label{e1.8}
\{\bar{a}_\la\bar{b}\} := \{a_\la b\} + \F^{m+k-2}V \in (\gr^{m+k-1} V)(\!(\lambda^{-1})\!) \,.
\end{equation}

The following is the main result of the paper.

\begin{theorem}\label{thm1.1}
For a filtered logarithmic vertex algebra\/ $V$, the associated graded\/ $\gr {V}$ has an induced structure of a non-local Poisson vertex algebra,
defined by \eqref{e1.5} and \eqref{e1.8}.
\end{theorem}

We refer to Theorem \ref{thm3.3} below for a more precise statement. In the rest of the paper, we work in the setting of superalgebras, since there are important examples
of logarithmic vertex superalgebras.
We use the relation between logarithmic vertex algebras and non-local Poisson vertex algebras to construct new examples of both non-local Poisson vertex algebras and logarithmic vertex algebras.

As one new example, we build a logarithmic vertex algebra associated to the so-called \emph{potential Virasoro--Magri} non-local Poisson vertex algebra \cite{DK}.
From it we obtain, in particular, the following non-linear extension of the Virasoro Lie algebra:
\begin{align*}
[C,D] &= [C, L_m] = [D, L_m] = 0  \qquad\qquad (m,k\in\ZZ) \,, \\
[L_{m}, L_{k}] &= (m-k)\sum_{j= 0}^\infty (j+1)L_{m+k-j}D^{j}\\
+ & \,\delta_{m+k\ge0} \, (m - k)\bigl((m-k)^2 - m - k - 4\bigr) \binom{m+k + 3}{3} \frac{C}{96}D^{m+k}\,;
\end{align*}
see Section \ref{sec4.3} below for more details.

Let us point out that several questions remain open about
the relationship between non-local Poisson vertex algebras and logarithmic vertex algebras. 
First, for ordinary vertex algebras, there is a forgetful functor to Lie conformal algebras:
every vertex algebra is a Lie conformal algebra with the $\la$-bracket \eqref{e1.3}. Furthermore, the axioms of a vertex algebra can be formulated
as a Lie conformal algebra with an additional product $\mu_{(-1)}$, so that they generalize those of a Poisson vertex algebra \cite{BK}.
We do not have an analogue of this for logarithmic vertex algebras. 
In fact, a logarithmic vertex algebra $V$ is not a non-local Lie conformal algebra with the $\la$-bracket \eqref{e1.7}; 
only its associated graded $\gr V$ is. 

A related question is how to generate a logarithmic vertex algebra from a non-local Lie conformal algebra, perhaps with some additional structure. In particular, we do not have a construction of a universal logarithmic vertex algebra $V(R)$ generated by a non-local Lie conformal algebra $R$. 

Finally, there exist examples of non-local Poisson vertex algebras (e.g.\ the non-linear Schr\"odinger from \cite{DK}), in which the $\la$-bracket is given by 
\eqref{e1.7} but for a braiding map $\N$ that does not satisfy all required axioms. This suggests that logarithmic vertex algebras could be generalized further by relaxing some of the conditions on $\N$. We will address this question in a forthcoming publication.

This paper is organized as follows. In Section \ref{sec2}, we give short introductions to non-local Lie conformal algebras, non-local Poisson vertex algebras, and logarithmic vertex algebras.

We start Section \ref{sec3} by defining filtered logarithmic vertex algebras and their associated graded, and stating the main result of the paper (Theorem \ref{thm3.3}).
The rest of the section is devoted to the proof of Theorem \ref{thm3.3} as follows. 
First, in Section \ref{sec3.2},  we  prove that the associated graded is a commutative associative unital differential algebra with the product \eqref{e1.5}.
Next, in Section \ref{sec3.3}, we prove that the $\lambda$-bracket \eqref{e1.8} satisfies the Leibniz rule and a few other properties. 
Last, the most difficult part is the proof of the Jacobi identity in Section \ref{sec3.4}.

In Section \ref{sec4}, we present three examples of non-local Poisson vertex algebras.
The first two are obtained as associated graded of logarithmic vertex algebras from \cite{BV}. 
The third example is the above-mentioned potential Virasoro--Magri non-local Poisson vertex algebra, 
which we realize as the associated graded of a new logarithmic vertex algebra.

Throughout the paper, we denote by $\ZZ_+$ the set of non-negative integers, and we use the divided-powers notation $x^{(k)}=x^k/k!$ for $k\ge0$; $x^{(k)}=0$ for $k<0$. We will work with vector superspaces over $\CC$. 
For a vector superspace $V=V_{\bar0}\oplus V_{\bar1}$ and a vector $v\in V_{\alpha}$, where $\alpha\in \ZZ/2\ZZ=\{\bar0,\bar1\}$, we denote the parity of $v$ by $p_v=\alpha$. Unless otherwise specified, all tensor products and $\Hom$'s are over $\CC$.
We denote the identity operator by $I$.

\section{Preliminaries}\label{sec2}
We start this section with a brief review of some notation on superspaces. Then we recall the definition and examples of non-local Poisson vertex algebras following \cite{DK}. 
Finally, we give a short introduction to logarithmic vertex algebras following \cite{BV}.

\subsection{Vector superspaces}\label{sec2.1}

Recall that the tensor product of two superspaces $V=V_{\bar0}\oplus V_{\bar1}$ and $W=W_{\bar0}\oplus W_{\bar1}$ is the superspace $V\otimes W$ with
\begin{equation*}
(V\otimes W)_{\alpha}=\bigoplus_{\beta\in \ZZ/2\ZZ}V_{\beta}\otimes W_{\alpha-\beta}\,,
\qquad \alpha\in \ZZ/2\ZZ = \{\bar0,\bar1\} \,.
\end{equation*}
To make formulas more compact, we will denote the parity of $a$ as $p_a$ and write $p_{a,b}$ for the product of parities $p_{a}p_{b}$.
We define the transposition operator on the superspace $V\otimes V$ as
\begin{equation}\label{logf1}
P(a\otimes b)=(-1)^{p_{a,b}}b\otimes a\, , \qquad a, b\in V\,.
\end{equation}
 As customary when writing such identities, by linearity, we assume that the elements $a$ and $b$ are homogeneous with respect to parity.

The vector space $\End(V)$, consisting of all linear operators on $V$, is naturally a superspace with 
\begin{equation*}
\End(V)_{\alpha}=\bigl\{\varphi\in \End(V)  \,\big|\, \varphi V_{\beta}\subset V_{\alpha+\beta} \;\; \forall\beta\in \ZZ/2\ZZ\bigr\}\,, \qquad \alpha\in \ZZ/2\ZZ\,,
\end{equation*}
and is an associative superalgebra with the product given by composition.
For any associative superalgebra, the commutator bracket 
\begin{equation*}
[a,b]:=ab-(-1)^{p_{a,b}}ba
\end{equation*}
defines the structure of a Lie superalgebra. 

Throughout the paper, we will work with an even endomorphism
\[\N\in \End(V)\otimes \End(V)\, .\] 
We can write $\N$ as a finite sum
\begin{equation}\label{logf-n}
\N=\sum_{i=1}^{L}\phi_{i}\otimes  \psi_{i}\, , \qquad \phi_{i}, \psi_{i}\in \End(V) \,, \quad p_i:=p_{\phi_i}=p_{\psi_i} \,,
\end{equation}
where $\phi_1,\dots,\phi_L$ are linearly independent
and $\psi_1,\dots,\psi_L$ are linearly independent.
The action of $\N$ on $V\otimes V$ is given by:
\begin{equation}\label{logf-nact}
\N(a\otimes b)=\sum_{i=1}^{L}(-1)^{p_{i,a}}\phi_{i}(a)\otimes  \psi_{i}(b)\, , \qquad a,b\in V\, .
\end{equation}

Given $\N$ as above, we define the endomorphisms $\N_{12}, \N_{23},\N_{13}$ $\in$ $\End(V)$ $\otimes \End(V)\otimes \End(V)$ as follows:
\begin{equation}\label{logf-1}
\begin{split}
\N_{12}&=\N\otimes I = \sum_{i=1}^{L} \phi_{i} \otimes \psi_{i} \otimes I \, , \\
\N_{23}&=I\otimes\N = \sum_{i=1}^{L} I \otimes \phi_{i} \otimes \psi_{i} \, , \\
\N_{13}&=(I\otimes P) \N_{12} (I\otimes P) = \sum_{i=1}^{L} \phi_{i} \otimes I \otimes \psi_{i} \,,
\end{split}
\end{equation}
where $I$ denotes the identity operator and $P$ is the transposition \eqref{logf1} on $V\otimes V$.
We will also use the notation $P_{12}=P \otimes I$.

\subsection{Non-local Poisson vertex algebras}\label{sec2.2}

In this subsection, we introduce the definitions of admissible brackets, non-local Lie conformal algebras, and non-local Poisson vertex algebras. This Section is based on \cite{DK}, except for Examples \ref{ex3.2}, \ref{ex-gl}, \ref{ex2.12.1}, and \ref{ex2.13}, which seem new.

\subsubsection{Admissible brackets}

Given a complex vector superspace $\mathcal{V}$, we consider the superspace 
\begin{equation}\label{Vlamu}
\mathcal{V}_{\lambda, \mu} :=\mathcal{V}[[\lambda^{-1},\mu^{-1},(\lambda+\mu)^{-1}]][\lambda, \mu] \,,
\end{equation}
where $\la$ and $\mu$ are even formal variables.
We refer to \cite[Section 2.1]{DK} for a detailed description of $\V_{\la,\mu}$.
In particular, we have natural embeddings
\begin{equation}\label{iota}
\begin{split}
\iota_{\mu, \lambda}\colon\mathcal{V}_{\lambda, \mu} &\rightarrow \mathcal{V}(\!(\lambda^{-1})\!)(\!(\mu^{-1})\!)\, ,\\
\iota_{\lambda, \mu}\colon\mathcal{V}_{\lambda, \mu} &\rightarrow \mathcal{V}(\!(\mu^{-1})\!)(\!(\lambda^{-1})\!)\, ,\\
\iota_{ \lambda, \lambda+\mu}\colon\mathcal{V}_{\lambda, \mu} &\rightarrow \mathcal{V}(\!((\lambda+\mu)^{-1})\!)(\!(\lambda^{-1})\!)\, ,
\end{split}
\end{equation}
defined by expanding one of the variables $\lambda$, $\mu$ or $\lambda+\mu$ in terms of the other two. More explicitly, in these expansions we let for $n<0$:
\begin{align*}
\iota_{\mu, \lambda}(\lambda+\mu)^{n}&=\sum_{k\in \mathbb{Z}_{+}}\binom{n}{k}\mu^{n-k}{\lambda^{k}}\, , \\
\iota_{\lambda, \mu}(\lambda+\mu)^{n}&=\sum_{k\in \mathbb{Z}_{+}}\binom{n}{k}\lambda^{n-k}{\mu^{k}}\, , \\
\iota_{\lambda, \lambda+\mu}\mu^{n}&=\sum_{k\in \mathbb{Z}_{+}}\binom{n}{k}(-1)^{n-k}\lambda^{n-k}{(\lambda+\mu)^{k}}\,.
\end{align*}

 
The main object in this section is an even linear map
\begin{equation*}
\{ \cdot_{\lambda} \cdot \} \colon\mathcal{V}\otimes\mathcal{V}\rightarrow\mathcal{V}(\!(\lambda^{-1})\!) \, ,
\end{equation*}
called the $\lambda$-\emph{bracket}. For any $a,b\in \V$, we can write their $\lambda$-bracket in the form
\begin{equation}\label{e2.8}
\{a_{\lambda}b\}=\sum_{n=-\infty}^{N}c_{n}\lambda^n
\end{equation}
for some $N\in \ZZ_{+}$ and $c_{n}\in\V$.
In general, when we compose such brackets, we obtain elements in the following spaces: 
\begin{align*}
\{a_{\lambda}\{b_{\mu}c\}\} &\in \mathcal{V}(\!(\lambda^{-1})\!)(\!(\mu^{-1})\!)\, ,\\
\{b_{\mu}\{a_{\lambda}c\}\} &\in \mathcal{V}(\!(\mu^{-1})\!)(\!(\lambda^{-1})\!)\, ,\\
\{\{a_{\lambda}b\}_{\lambda+\mu}c\} &\in \mathcal{V}(\!((\lambda+\mu)^{-1})\!)(\!(\lambda^{-1})\!)\, ,
\end{align*}
where $a,b,c\in \V$. 
In order to be able to express identities relating these elements, they have to belong in the same space,
which leads to the next definition.
\begin{definition}\label{admisible} We say that the $\la$-bracket $\{ \cdot_{\lambda} \cdot\}$ is \emph{admissible} if for all $a,b,c\in \mathcal{V}$ we have:
\begin{align}
\label{adm1}
\{a_{\lambda}\{b_{\mu}c\}\} &\in \iota_{\mu, \lambda} \mathcal{V}_{\lambda,\mu}\, ,\\
\label{adm2}
\{b_{\mu}\{a_{\lambda}c\}\} &\in \iota_{\lambda, \mu} \mathcal{V}_{\lambda,\mu}\, ,\\
\label{adm3}
\{\{a_{\lambda}b\}_{\lambda+\mu}c\} &\in \iota_{\lambda, \lambda+\mu} \mathcal{V}_{\lambda,\mu}\, .
\end{align}
\end{definition}

\subsubsection{Non-local Lie conformal algebras}
Consider a $\CC[\partial]$-module $\V$, i.e., a vector superspace equipped with an even linear map 
$\partial\colon \V\rightarrow \V$.
Below, we will always expand in non-negative powers of $\partial$ as follows:
\begin{equation}\label{e2.9e}
(\lambda+\partial)^{n}:=\sum_{k\in \mathbb{Z}_{+}}\binom{n}{k}\lambda^{n-k}{\partial^{k}} \,.
\end{equation}
For example, from \eqref{e2.8}, we have that 
\begin{equation}\label{e2.9}
\begin{split}
\{a_{-\lambda-\partial}b\}
&=\sum_{n=-\infty}^{N}(\lambda+\partial)^{n} (-1)^n c_{n} \\
&=\sum_{n=-\infty}^{N}\sum_{k\in \mathbb{Z}_{+}}\binom{n}{k}(-1)^n \lambda^{n-k}{\partial^{k}}c_{n}\\
&=\sum_{m=-\infty}^{N}\left(\sum_{k=0}^{N-m}\binom{m+k}{k}(-1)^{m+k} \partial^{k}c_{m+k}\right)\lambda^{m}\, .
\end{split}
\end{equation}
Hence, $\{a_{-\lambda-\partial}b\}\in \V(\!(\lambda^{-1})\!)$.

\begin{definition}\label{def2.2}
A \emph{non-local Lie conformal algebra} (abbreviated non-local LCA)  is a $\mathbb{C}[\partial]$-module $\mathcal{V}$ equipped with an admissible $\la$-bracket
$$\{ \cdot_{\lambda}\cdot \} \colon\mathcal{V}\otimes\mathcal{V}\rightarrow\mathcal{V}(\!(\lambda^{-1})\!)\, ,$$
 satisfying the following axioms:
\begin{align*}
\text{(sesqui-linearity)}&\quad \{\partial a_{\lambda}b\}=-\lambda \{a_{\lambda}b\}\, ,\quad\{a_{\lambda}\partial b\}=(\lambda+\partial)\{a_{\lambda} b\}\, ,\\
\text{(skew-symmetry)}&\quad \{b_{\lambda}a\}=-(-1)^{p_{a,b}} \{a_{-\lambda-\partial}b\}\, ,\\
\text{(Jacobi identity)}&\quad \{a_{\lambda}\{b_{\mu}c\}\}=\{\{a_{\lambda}b\}_{\lambda+\mu}c\}+(-1)^{p_{a,b}}\{b_{\mu}\{a_{\lambda}c\}\}\, .
\end{align*}
\end{definition}

In the skew-symmetry identity above we use \eqref{e2.9}, and the Jacobi identity is viewed as an identity in $\V_{\lambda,\mu}$ after we identify $\V_{\lambda,\mu}$ with its image under the embeddings \eqref{iota}; see Definition \ref{admisible}. 
Note that, by sesqui-linearity, the $\la$-bracket is uniquely determined by its values on a set of generators of $\V$ as a $\C[\partial]$-module.

\begin{remark}\label{rem2.1} 
In \cite{DK}, the authors initially use a weaker definition of admissible brackets, which requires only \eqref{adm1}. It was shown in \cite[Remark 3.3]{DK} that every skew-symmetric $\la$-bracket satisfying \eqref{adm1} must satisfy \eqref{adm2} and \eqref{adm3} as well; hence, it is admissible according to our definition.
\end{remark}

Note that if the $\la$-bracket takes values in $\V[\la] \subset \V(\!(\lambda^{-1})\!)$, then the notion of a non-local Lie conformal algebra coincides with that of a Lie conformal algebra \cite{K1}.
In particular, any $\C[\partial]$-module with the trivial $\la$-bracket $\{ \cdot_{\lambda}\cdot \} = 0$ is a (non-local) Lie conformal algebra. Here are a few more examples.


\begin{example}[Potential free boson LCA]\label{cur2} The $\C[\partial]$-module 
\[\V=\bigl(\C[\partial]\otimes \C x \bigr) \oplus \C K\, ,\qquad\partial K=0\, ,\]
is a non-local Lie conformal algebra with $\lambda$-bracket defined by  
\begin{equation}\label{equ2.11}
\{x_{\lambda}x\}=-\frac{1}{\lambda}K \,, \qquad \{K_\la\V\} = 0 \, .
\end{equation}
Then, by sesqui-linearity, $\partial x =x'$ satisfies the $\lambda$-bracket of the free boson or Heisenberg Lie conformal algebra:
\begin{equation}\label{equ2.11b}
\{x'_{\lambda}x'\}=\lambda K \, .
\end{equation}
\end{example}

\begin{example}[Potential Virasoro--Magri LCA]\label{ex3} The $\C[\partial]$-module 
\[\bigl(\C[\partial]\otimes \C u\bigr) \oplus \C C\, ,\quad\partial C=0\, ,\]
is a non-local Lie conformal algebra with $\lambda$-bracket defined by  
\begin{equation}\label{equ2.12}
\{u_{\lambda}u\}=-\Bigl(\frac{1}{\lambda+\partial}+\frac{1}{\lambda}\Bigr) u'-\frac{\lambda}{12}C\, ,
\end{equation}
where $\partial u =u'$. Note that, by sesqui-linearity, $u'$ satisfies the relations of the Virasoro Lie conformal algebra:
\begin{equation}\label{equ2.12.1}
\{u'_{\lambda}u'\} 
=(2\lambda+\partial)u'+\frac{\lambda^{3}}{12}C\, .
\end{equation}
\end{example}

Motivated by the previous example, 
we consider a similar construction for affine Lie conformal algebras.

\begin{example}[Potential affine LCA]\label{ex3.2} Let $\lieg$ be a Lie algebra with a non-degenerate symmetric invariant bilinear form $(\cdot\,|\,\cdot )$.  The $\C[\partial]$-module 
\[\bigl(\C[\partial]\otimes \lieg\bigr) \oplus \C K\, ,\quad\partial K=0\, ,\]
is a non-local Lie conformal algebra with $\lambda$-bracket defined by  
\begin{equation}\label{equ2.12.a}
\{a_{\lambda}b\}=\Bigl(\frac{1}{\lambda+\partial}-\frac{1}{\lambda}\Bigr)[a,b] - \frac{1}{\lambda} (a|b) K\,.
\end{equation}
Note that, by sesqui-linearity, $\partial a=a'$ and $\partial b=b'$ satisfy the relations of the affine Lie conformal algebra:
\begin{equation}\label{equ2.12.b}
\{ a'_{\lambda} b'\}=[a, b]'+\lambda (a|b) K\, .
\end{equation}
\end{example}

Our last example is more intricate, and its $\lambda$-brackets will be derived in Section \ref{ex4.3} below.
The fact that they satisfy the axioms of a non-local LCA is a non-trivial consequence of our main theorem (Theorem \ref{thm3.3}).

\begin{example}[Gurarie--Ludwig LCA]\label{ex-gl}
The $\C[\partial]$-module 
\begin{equation*}
\C[\partial]\otimes \bigl(\C L +\C\xi+\C\bar{\xi}+\C\ell \bigr) \oplus \C K,\qquad\partial K=0 \,,
\end{equation*}
with even variables $L, \ell,K$ and odd variables $\xi, \bar{\xi}$, is a non-local Lie conformal algebra with the following $\lambda$-brackets 
on the generators:  
\begin{equation}\label{equ2.14}
\begin{split}
\{L_{\lambda}L\}&=(2\lambda+\partial)L\,, \qquad\qquad
\{L_{\lambda}\ell\} =(2\lambda+\partial)\ell+ \frac{\lambda^3}{6} K\, , \\
\{L_{\lambda}\xi\}&=(2\lambda+\partial)\xi\, , \qquad\qquad\,
\{L_{\lambda}\bar{\xi}\}=(2\lambda+\partial)\bar{\xi}\, , \\
\{\ell_{\lambda}\xi\}&= \Bigl(\frac{1}{\lambda+\partial}\xi\Bigr)L\, , \qquad\quad\;\;
\{\ell_{\lambda}\bar{\xi}\} = \Bigl(\frac{1}{\lambda+\partial}\bar{\xi}\Bigr)L\, , \\
\{\ell_{\lambda}\ell\}&=2\Bigl(\frac{1}{\lambda+\partial}\xi\Bigr)\bar{\xi}-2\Bigl(\frac{1}{\lambda+\partial}\bar{\xi}\Bigr){\xi}\,,  \\
\{\xi_{\lambda}\bar{\xi}\}&=\frac{1}{2}(\partial+2\lambda) \ell + \frac{\lambda^3}{12} K - \frac{1}{2}\Bigl(\frac{1}{\lambda+\partial}L\Bigr)L \,, 
\end{split}
\end{equation}
where all other brackets are either zero or obtained from skew-symmetry.
\end{example}

\subsubsection{Non-local Poisson vertex algebras}
We now present one of the main definitions from \cite{DK}.

\begin{definition}\label{def2.9}
A \emph{non-local Poisson vertex algebra} (abbreviated non-local PVA)  is a quintuple $(\mathcal{V},1, \cdot , \partial ,\{ \cdot_{\lambda}\cdot\})$ satisfying the following axioms:
\begin{enumerate}
\smallskip
\item[(1)] $(\mathcal{V},\partial,\{ \cdot_{\lambda}\cdot \})$ is a non-local Lie conformal algebra;
\smallskip
\item[(2)] $(\mathcal{V}, 1, \partial, \cdot )$ is a commutative associative unital differential algebra;
\smallskip
\item[(3)]\label{de3} $\{a_{\lambda}bc\}=\{a_{\lambda}b\}c+(-1)^{p_{a,b}} b\{a_{\lambda}c\}$ for all $a,b,c\in\mathcal{V}$ (Leibniz rule).
\end{enumerate}
\end{definition}

Now, we describe examples of non-local PVA's. These examples appear as quotients of the symmetric algebras over the Lie conformal algebras in Examples \ref{cur2}--\ref{ex-gl}, with $\partial$ extended as a derivation and the bracket $\{\cdot_{\lambda}\cdot\}$ extended by the Leibniz rule. We quotient by the relation that identifies the central element $K$ or $C$ (if present) with a certain constant multiple of $1$.


\begin{example}[Potential free boson PVA]\label{ex2.11} The $\C[\partial]$-module 
\[\C[x,x',x'',\dots] \]
is a non-local PVA with $\lambda$-bracket $\{x_\lambda x\}$ defined by  \eqref{equ2.11} with $K=1$.
\end{example}

\begin{example}[Potential Virasoro--Magri PVA]\label{ex2.12} For a fixed $c\in\C$ (called the central charge), we have the non-local PVA
\[\C[u,u',u'',\dots] ,\]
where the $\lambda$-bracket $\{u_\lambda u\}$ is defined by \eqref{equ2.12} with $C=c1$.
\end{example}

\begin{example}[Potential affine PVA]\label{ex2.12.1}  Let $\lieg$ be a Lie algebra with a non-degenerate symmetric invariant bilinear form $(\cdot\,|\,\cdot )$,
and a basis $\{a_1, \dots , a_N\}$.
For a fixed $k\in\C$ (called the level), we have the non-local PVA
\begin{equation*}
\CC[a_1, \dots , a_N , a'_1, \dots , a'_N , a''_1, \dots , a''_N , \dots ]
\end{equation*}
with $\lambda$-bracket defined by \eqref{equ2.12.a} with $K=k1$.
\end{example}

\begin{example}[Gurarie--Ludwig PVA]\label{ex2.13} For a fixed $\beta\in\C$, the $\C[\partial]$-module 
\begin{equation*}
\CC[L, \xi, \bar{\xi}, \ell, L', \xi', \bar{\xi}', \ell' , L'', \xi'', \bar{\xi}'', \ell'' ,\dots ]
\end{equation*}
is a non-local PVA with $\lambda$-bracket defined by \eqref{equ2.14} with $K=\beta 1$.
\end{example}

As in \cite{DK}, we will use the notation
\begin{equation}\label{earr}
\{a_{\lambda+\partial}b\}_{\rightarrow}c:=\sum_{n=-\infty}^{N}c_{n}(\lambda+\partial)^nc
\end{equation}
for $\{a_{\lambda}b\}$ given by \eqref{e2.8}. Then the skew-symmetry and Leibniz rule imply the following \emph{right Leibniz rule}:
\begin{equation}\label{lm2.6}
\{ab_{\lambda}c\}=(-1)^{p_{b,c}}\{a_{\lambda+\partial}c\}_{\rightarrow}b+(-1)^{p_{a,b}+p_{a,c}}\{b_{\lambda+\partial}c\}_{\rightarrow}a\, .
\end{equation}

\subsection{Logarithmic vertex algebras}\label{sec2.3}

In this subsection, following \cite{BV}, we review the definitions of logarithmic fields, braiding maps, logarithmic vertex algebras, and some of their main properties.

\subsubsection{Logarithmic fields and braiding maps}
We will denote by $\z,\z_1,\z_2,\dots$ and $\ze,\ze_1,\ze_2,\dots$ commuting even formal variables, 
and we will use the notation $z_{ij}=z_i-z_j$.
The variables $\ze,\ze_1,\ze_2,\dots$ are thought of as $\ze=\log\z$ and $\ze_i=\log\z_i$. We consider the operators
\begin{equation*}
D_\z = \d_\z+\z^{-1} \d_\ze\,, \qquad D_{\z_i} = \d_{\z_i}+\z_i^{-1} \d_{\ze_i}\,.
\end{equation*}
Note that $D_\z$ is the total derivative with respect to $\z$ if we set $\ze=\log\z$. We will also need the formal power series
\begin{equation}\label{eq2.10}
\vartheta_{12}:=\zeta_{1}-\sum_{n=1}^\infty \frac{1}{n}z_{1}^{-n} z_{2}^{n} \,,
\end{equation}
which can be thought of as an expansion of $\log (z_{12}) = \log z_1+\log(1-\frac{z_2}{z_1})$.
Similarly, we define
\begin{equation}\label{eq2.11}
\vartheta_{21}:=\zeta_{2}-\sum_{n=1}^\infty \frac{1}{n}z_{1}^{n} z_{2}^{-n} \,,
\end{equation}
which is an expansion of $\log (z_{21})$.

Now we introduce the notion of a logarithmic field.

\begin{definition}\label{d2.1}
Let $V$ be a vector superspace.
The superspace 
\begin{equation*}
\LF(V) = \LF(V)_{\bar0} \oplus \LF(V)_{\bar1}
\end{equation*}
of \emph{logarithmic (quantum) fields} is defined by
\begin{equation*}
\LF(V)_\alpha = \Hom(V_{\bar0},V_{\alpha}(\!(\z)\!)[\ze]) \oplus \Hom(V_{\bar1},V_{\alpha+\bar1}(\!(\z)\!)[\ze]) \,,
\qquad \al\in\ZZ/2\ZZ \,.
\end{equation*}
In particular, when $V$ is a (purely even) vector space, we have
\begin{equation*}
\LF(V) = \Hom(V,V(\!(\z)\!)[\ze]) \,.
\end{equation*}
We denote the elements of $\LF(V)$ as $a(z,\ze)$ or $a(\z)$ for short. 
\end{definition}

In the above definition, $V(\!(\z)\!)[\ze] = V[[z]][z^{-1},\ze]$ stands for the space of polynomials in $\ze$ whose coefficients are
formal Laurent series in $z$.

\begin{definition}\label{lm}
Let $V$ be a vector superspace. An (infinitesimal) \emph{braiding map} $\N$ on $V$ is an even linear operator
\begin{equation*}
\N\in \End(V)\otimes \End(V)\, ,
\end{equation*}
satisfying the following two conditions:
\begin{enumerate}

\item\label{2.4-2}
$\N$ is \emph{symmetric}, i.e., $\N P=P \N\,$ where $P$ is the transposition \eqref{logf1};

\smallskip
\item\label{2.4-3}

$[\N_{12},\N_{23}]=[\N_{13},\N_{23}]=[\N_{12},\N_{13}]=0$, where $\N_{ij}$ are as in \eqref{logf-1}.
\end{enumerate}
\end{definition}

Finally, we say that $\N$ is \emph{locally nilpotent} on $V \otimes V$ if for all $a,b\in V$ there exists a positive integer $r$ such that  
$\N^r (a\otimes b) = 0$. 

\subsubsection{Logarithmic vertex algebras}\label{sec2.3.2}

\begin{definition} \label{d3.11}
A \emph{logarithmic vertex algebra} (abbreviated logVA) 
is a vector superspace $V$, equipped with an even vector $\vac\in V_{\bar 0}$, an even endomorphism $T\in \End(V)_{\bar 0}$,  an even linear map 
\[Y\colon  V\to \LF(V)\,, \qquad a\mapsto Y(a,z)\,, \]
and a braiding map $\N$ on $V$ (see \deref{lm}),
subject to the following axioms:

\medskip
(\emph{vacuum})\; $Y(\vac,z)=I$, \; $Y(a,z)\vac\in V[\![z]\!]$, \; $Y(a,z)\vac\big|_{z=0} = a$, \; $T\vac=0$.

\medskip
(\emph{translation covariance})\; 
$[T,Y(a,z)] = D_z Y(a,z)$.

\medskip
(\emph{nilpotence})\; $\N$ is locally nilpotent on $V\otimes V$.

\medskip
(\emph{locality})\;  For every $a,b\in V$, there exists $N\in\ZZ_+$ such that for all $c\in V$, 
\begin{equation}\label{eq2.12}
\begin{split}
Y(z_1)& (I \otimes Y(z_2)) z_{12}^{N}e^{ \vartheta_{12}\N_{12}} (a\otimes b\otimes c )\\
&=(-1)^{p_{a,b}} \, Y(z_2) (I \otimes Y(z_1)) z_{12}^{N}e^{\vartheta_{21} \N_{12}} (b\otimes a\otimes c ),
\end{split}
\end{equation}
where we use the notation $Y(z)(a\otimes b):=Y(a,z)b$.

\medskip
(\emph{hexagon})\; The following identity is satisfied on $V^{\otimes 3}$: 
\begin{equation} \label{eq2.13}
\N (Y(z)\otimes I) = (Y(z)\otimes I) (\N_{13}+\N_{23})\,,
\end{equation}
where we use the notation \eqref{logf-1}.
\end{definition}

We denote a logVA by $(V,\vac,T, Y,\N)$ or simply $V$ for short.
We will provide some examples of logarithmic vertex algebras in Section \ref{sec4} below.

\subsubsection{Properties of logVAs}
Now we discuss some properties of logVAs. In any logVA\/ $V$, for each vector $a\in V$,  we have a family of linear operators $a_{(n+\N)} \in\End(V)$, called the \emph{modes} of $a$, given by 
\begin{equation}\label{logva1} 
X(a,z):=Y(a,z)\big|_{\ze=0} = \sum_{n\in \ZZ} a_{(n+\N)} \, z^{-n-1} \,, \qquad a\in V\,.
\end{equation}
These define $\CC$-bilinear maps called $(n+\N)$-\emph{products}:
\begin{equation}\label{eq3.7a}
\mu_{(n)} \colon V\otimes V\rightarrow V\,, \quad {\mu}_{(n)}(a\otimes b)= a_{(n+\N)}b\, , \qquad n\in \ZZ\, .
\end{equation}
The logarithmic fields in a logVA can be expressed in terms of $(n+\N)$-products as follows
\cite[Proposition 3.7]{BV}:
\begin{equation}\label{p-modes1}
\begin{split}
Y(a,z)b &= \sum_{n\in \ZZ} \mu_{(n)} \bigl( z^{-n-1-\N} (a \otimes b) \bigr) \\
&=\sum_{ \substack{n\in\ZZ \\ i\in\ZZ_+} } \frac{(-1)^{i}}{i!} \ze^i z^{-n-1} \mu_{(n)} \bigl( \N^i (a \otimes b) \bigr) \,,
\end{split}
\end{equation}
where $z^{-\N}:=e^{-\ze\N}$.


The $(n+\N)$-products satisfy an analogue of the \emph{Borcherds identity} for all $k,m,n\in\ZZ$: 
\begin{equation}\label{borcherds2}
\begin{split}
\sum_{j\in \ZZ_+} &(-1)^{j}\mu_{(m+n-j)}(I\otimes \mu_{(k+j)})\binom{n+\N_{12}}{j}\\
& -\sum_{j\in \ZZ_+}(-1)^{n+j}\mu_{(n+k-j)}(I\otimes \mu_{(m+j)})\binom{n+\N_{12}}{j} P_{12} \\
& =\sum_{j\in \ZZ_+} \mu_{(m+k-j)}( \mu_{(n+j)}\otimes I) \binom{m+\N_{13}}{j}\,,
\end{split}
\end{equation}
where $P_{12} = P \otimes I$ and $P$ is the transposition \eqref{logf1}.
This is an identity on $V^{\otimes 3}$, and is proved in \cite[Theorem 3.26]{BV}.

We will also need the following \emph{skew-symmetry} relation in logVAs
\cite[Proposition 3.22]{BV}:
\begin{equation}\label{ss1}
Y(a,z,\ze) b = (-1)^{p_{a,b}} e^{zT}Y(b,-z,\ze)a \,.
\end{equation}
In particular, setting $\zeta=0$ in \eqref{ss1}, we have
\begin{equation}\label{ss2}
X(a,z) b = (-1)^{p_{a,b}} e^{zT}X(b,-z)a \,.
\end{equation}

It will be useful to express some of the properties of the fields $Y(a,z)$ or $X(a,z)$
as properties of the $(n+\N)$-products.

\begin{proposition}\label{co}
For any logVA\/ $V$ and\/ $a,b\in V$, $n\in\ZZ$, we have the following relations{\rm:}
\begin{enumerate}
\medskip

\item \label{cor2.1}
$\mu_{(-1)}(\vac \otimes a)=\mu_{(-1)}(a\otimes \vac)=a;$
\medskip

\item \label{cor2.2}
$\mu_{(n)}(Ta \otimes b)=-n\mu_{(n-1)}(a \otimes b)-\mu_{(n-1)}(\N(a\otimes b));$
\medskip

\item \label{cor2.3}
$T\,\mu_{(n)}(a\otimes b)=\mu_{(n)}(Ta \otimes b)+\mu_{(n)}(a\otimes Tb);$
\medskip

\item \label{cor2.4}
$\mu_{(n)}( a\otimes b)=- (-1)^{p_{a,b}} \displaystyle\sum_{j\in\ZZ_+}(-1)^{j+n}\frac{1}{j!}T^{j}\mu_{(n+j)}(b\otimes a)$.
\end{enumerate}
\end{proposition}

\begin{proof} 

\eqref{cor2.1} follows from \eqref{logva1} and the vacuum axiom in Definition \ref{d3.11}. 

\eqref{cor2.2} follows from the translation covariance axiom in Definition \ref{d3.11} and the mode expansion \eqref{p-modes1}.

\eqref{cor2.3} means that $T$ is a derivation of all $(n+\N)$-products. By translation covariance, this is equivalent to
$D_{z}Y(a,z)=Y(Ta,z)$,
and is proved in \cite[Proposition 3.5]{BV}.  

\eqref{cor2.4} follows from \eqref{logva1} and the skew-symmetry \eqref{ss2}, as in the case of ordinary vertex algebras (see \cite{K1}).
\end{proof}

Similarly, we express some properties of the braiding map $\N$ as
identities for the endomorphisms $\phi_{i}, \psi_{i}\in \End(V)$,  $1\leq i\leq L$, defined in \eqref{logf-n}.

\begin{proposition}\label{co3}
For any logVA\/ $V$ and\/ $a,b\in V$, we have{\rm:}
\begin{enumerate}
\medskip

\item \label{co3.11}
$\displaystyle\sum_{i=1}^{L}\phi_{i}\otimes  \psi_{i}=\sum_{i=1}^{L} (-1)^{p_i} \psi_{i}\otimes  \phi_{i} \,;$
\medskip

\item \label{co3.2}
$[\phi_{i}, \phi_{j}]=[\phi_{i}, \psi_{j}]=[\psi_{i}, \psi_{j}]=0;$
\medskip

\item \label{co3.3}
$\phi_{i}\,\mu_{(n)}(a\otimes b)=\mu_{(n)}(\phi_{i}a \otimes b)+(-1)^{p_{i,a}} \mu_{(n)}(a\otimes \phi_{i}b);$
\medskip

\item \label{co3.4}
$\psi_{i}\,\mu_{(n)}(a\otimes b)=\mu_{(n)}(\psi_{i}a \otimes b)+(-1)^{p_{i,a}} \mu_{(n)}(a\otimes \psi_{i}b);$
\medskip

\item \label{co3.1}
$[T, \phi_{i}]=[T, \psi_{i}]=0.$
\end{enumerate}
\end{proposition}
\begin{proof} 
\eqref{co3.11} is equivalent to the symmetry of $\N$; see \deref{lm}\eqref{2.4-2}.

\eqref{co3.2} is equivalent to \deref{lm}\eqref{2.4-3}.

\eqref{co3.3} means that all $\phi_i$ are derivations of the $(n+\N)$-products, and is equivalent to the hexagon axiom in Definition \ref{d3.11}. 

\eqref{co3.4} follows from \eqref{co3.3} and the symmetry of $\N$.

\eqref{co3.1} follows from the fact that every derivation of a logVA commutes with $T$; see \cite[Remark 3.15]{BV}.  
\end{proof}

%
%
%
%
%
%
%

\section{Poisson Vertex Algebras as Associated Graded of Logarithmic Vertex Algebras}\label{sec3}

We start this section by introducing the notion of a filtered logarithmic vertex algebra and formulating the main theorem of this work. The rest of the section is devoted to the proof of the theorem. In Section \ref{sec3.2}, we show that the associated graded of a filtered logarithmic vertex algebra is a commutative associative unital differential algebra. In Section \ref{sec3.3}, we prove that the associated graded has an admissible $\lambda$-bracket that satisfies the sesqui-linearity, skew-symmetry, and Leibniz rule. Finally, in Section \ref{sec3.4},  we prove the Jacobi identity for the $\lambda$-bracket.

\subsection{Filtered logVAs and their associated graded}\label{sec3.1}
Let $(V,\vac,T, Y,\N)$ be a logVA, as in \deref{d3.11}.

\begin{definition}\label{fil} 
A \emph{filtration} of a logVA $V$ is an increasing exhaustive filtration by subspaces:
\[\{0\}=\cdots=\F^{-1}V\subset \F^{0}V\subset \F^{1}V\subset \F^{2}V\subset\cdots \,, \quad \bigcup_{n\in\ZZ_+}\F^{n}V=V \,,\]
satisfying the following conditions $(m,n\in\ZZ_+)$:
\begin{enumerate}
\medskip

\item $\vac\in \F^{0}V$;
\medskip

\item\label{1a} $T(\F^{n}V)\subset \F^{n}V$;
\medskip

\item\label{2a} $\N(\F^{m}V\otimes \F^{n}V)\subset \F^{m+n-1}(V\otimes V) := \displaystyle\sum_{k=0}^{m+n-1}\F^{k}V\otimes \F^{m+n-1-k}V$;
\medskip

\item $
    \mu_{(j)}(\F^{m}V\otimes \F^{n}V)\subset 
                \begin{cases}
                  \F^{m+n}V\, ,  & j<0 \,, \\
                  \F^{m+n-1}V\, ,  &  j\geq 0\, .
                \end{cases}
  $
  \end{enumerate}
\smallskip
A \emph{filtered} logVA is a logVA equipped with a filtration satisfying the above properties.
\end{definition}

For a filtered logVA $V$, we consider its associated graded $\gr V$, defined by 
\[\gr V:=\bigoplus_{n\in\ZZ_+} \gr^{n} V\, , \qquad \gr^{n} V:=\F^{n}V/\F^{n-1}V\, . \]
Note that $\gr^n V = \{0\}$ for $n<0$, and we can identify $\gr^0 V = \F^0 V$ since $\F^{-1}V=\{0\}$.
The image of $\vac\in\F^0V$ in $\gr^0V$ will be denoted simply as $1$, because it will be the unit in the algebra $\gr V$, as we will show below.
The maps $T$, $\N$, and $\mu_{(j)}$ $(j\in\ZZ)$ from the logVA structure of $V$ induce corresponding maps on $\gr V$, which will be denoted as
$\partial$, $\N$, and $\mu_{(j)}$, respectively.
In more details, we have the following induced linear maps:

\begin{enumerate}
\medskip
\item\label{1} $\partial(\gr^{n} V)\subset \gr^{n} V$;

\medskip
\item\label{2} $\N(\gr^{m} V\otimes \gr^{n} V)\subset \gr^{m+n-1} (V\otimes V) := \displaystyle\bigoplus_{k=0}^{m+n-1}\gr^{k} V\otimes \gr^{m+n-1-k}V$;

\medskip
\item\label{3} $
    \mu_{(j)}(\gr^{m}V\otimes \gr^{n}V)\subset
                \begin{cases}
                  \gr^{m+n}V\, ,  & j<0 \,, \\
                  \gr^{m+n-1}V\, ,  &  j\geq 0\, .
                \end{cases}
  $
\medskip
  \end{enumerate}
  

In particular, from these properties we derive the following useful consequence of \prref{co}.

\begin{corollary}\label{cor} 
For any filtered logVA\/ $V$, we have the following relations in\/ $\gr V$ $(a,b\in \gr V, \, n\in\ZZ){:}$
\begin{enumerate}

\item  \label{2cor1}
$\mu_{(-1)}(1 \otimes a)=\mu_{(-1)}(a\otimes 1)=\/a;$
\smallskip

\item \label{2-1cor5}
 $\mu_{(0)}(\partial a \otimes b)=-\mu_{(-1)}(\N (a \otimes b));$
\smallskip

\item \label{2cor5}
$\mu_{(n)}(\partial a \otimes b)=-n\mu_{(n-1)}(a \otimes b)$ \; for \; $n\ne 0;$
\smallskip

\item \label{2cor6}
$\partial\,\mu_{(n)}(a\otimes b)=\mu_{(n)}(\partial a \otimes b)+\mu_{(n)}(a\otimes \partial b);$
\smallskip

\item \label{2cor7}
$\mu_{(n)}( a\otimes b)=-(-1)^{p_{a,b}} \displaystyle\sum_{j\in\ZZ_+}(-1)^{j+n}\frac{1}{j!}\partial^{j}\mu_{(n+j)}(b\otimes a)$.
\end{enumerate}
\end{corollary}

Starting from a filtered logVA $V$,
we define a $\lambda$-bracket on $\gr V$ as follows ($a,b\in\gr V$):
\begin{equation}\label{eq2.1}
\{a_{\lambda}b\}:=[a_{\lambda}b]
+\mu_{(-1)}\Bigl(\N\, \Bigl(\frac{1}{\lambda+\partial }\,  a\otimes b\Bigr)\Bigr)\,,
\end{equation}
where
\begin{equation}\label{eq2.1b}
[a_{\lambda}b]:=\sum_{n\in\ZZ_+}\frac{\lambda^{n}}{n!}\mu_{(n)}(a\otimes b)\in  (\gr V)[\lambda]
\end{equation}
is similar to the usual $\lambda$-bracket for Lie conformal algebras.
Note that, due to the expansion convention \eqref{e2.9e} and the above properties \eqref{1}--\eqref{3}, we have:
\begin{equation}\label{eq2.1a}
\{a_{\lambda}b\} \in (\gr^{m+n-1} V) (\!(\lambda^{-1})\!) \quad\text{for}\quad a\in\gr^m V \,, \; b\in\gr^n V \,.
\end{equation}

Now we can state the main result of the paper.

\begin{theorem}\label{thm3.3}
For any filtered logarithmic vertex algebra\/ $V$, the above defined\/ $(\gr {V},1, \mu_{(-1)} , \partial ,\{ \cdot_{\lambda}\cdot\})$ is a non-local Poisson vertex algebra.
\end{theorem}

The rest of this section is devoted to the proof of the theorem. We verify each of the axioms in the definition of a non-local PVA (Definition \ref{def2.9}), as follows.
In 
Proposition \ref{pro3.5}, we show that $(\gr V, \partial, 1 , \mu_{(-1)})$ is a commutative associative unital differential algebra.
In 
Propositions \ref{propbrack}--\ref{propadm}, we prove that the $\la$-bracket $\{\cdot_{\lambda}\cdot\}$ satisfies the sesqui-linearity, skew-symmetry, Leibniz rule, and is admissible.
In 
Proposition \ref{pro3.11}, we prove that $\{\cdot_{\lambda}\cdot\}$ satisfies the Jacobi identity.

Recall that we write the braiding map $\N$ as in \eqref{logf-n} in terms of endomorphisms $\phi_{i}, \psi_{i}\in \End(V)$,  $1\leq i\leq L$.
As we do not assume that $\phi_{i}$ and $\psi_{i}$ preserve the filtration of $V$, in general they do not induce maps on the associated graded $\gr V$.
However, notice that they always come together because all formulas involve the braiding map $\N$, which does descend to $\gr V$.
For convenience, from now on we will continue to use $\phi_{i}$ and $\psi_{i}$ in identities on $\gr V$, with the understanding that
$\N = \sum \phi_i \otimes \psi_i$ is defined there.
With this notation, both
\[
\sum_{i=1}^L (-1)^{p_{i,a}} \phi_i a \otimes \psi_i b = \N(a \otimes b)
\]
and
\[
\sum_{i=1}^L (-1)^{p_{i,a}} (\phi_i a) (\psi_i b) = \mu_{(-1)}\bigl(\N(a \otimes b)\bigr)
\]
make sense for $a,b\in\gr V$, and the results of \prref{co3} still apply (with $\partial$ in place of $T$ in part \eqref{co3.1}).
Hence, by \eqref{logf-nact}, we can rewrite \eqref{eq2.1} in the convenient form
\begin{equation}\label{eq3.4}
\{a_{\lambda}b\}:=[a_{\lambda}b] + \sum_{i=1}^{L} (-1)^{p_{i,a}} \Bigl(\frac{1}{\lambda+\partial }\, \phi_{i} a\Bigr)(\psi_{i} b)
\end{equation}
for $a,b\in\gr V$. Expanding $(\lambda+\partial)^{-1}$ as in \eqref{e2.9e} and using \coref{cor}\eqref{2cor5}, we also have
\begin{equation}
\mu_{(-1)}\Bigl(\N\, \Bigl(\frac{1}{\lambda+\partial }\,  a\otimes b\Bigr)\Bigr)=\sum_{n\in\ZZ_+}(-1)^{n}n!{\lambda^{-n-1}}\mu_{(-n-1)}\bigl(\N(a\otimes b)\bigr).
\end{equation}


\subsection{Commutative associative unital differential superalgebra}\label{sec3.2}

\begin{proposition}\label{pro3.5} 
Let\/ $V$ be a filtered logVA. Then $(\gr V, \partial, 1 , \mu_{(-1)})$ is a commutative associative unital differential superalgebra.
\end{proposition}
\begin{proof} 
From part \eqref{2cor1} of \coref{cor}, we have that $1$ is a unity of $\mu_{(-1)}$, and from part \eqref{2cor6} we obtain that $\partial$ satisfies the Leibniz rule.

To prove the commutativity of $\mu_{(-1)}$,
we apply the Borcherds identity \eqref{borcherds2} for $n=0$ and $m=k=-1$. 
Recall that as maps on the associated graded, $\N$ and $\mu_{(r)}$ for $r\ge0$ have degree $-1$, and $\mu_{(r)}$ for $r<0$ have degree $0$.
Hence, all terms of \eqref{borcherds2} induce corresponding maps on $(\gr V)^{\otimes 3}$. 
If we restrict them to $\gr^l(V^{\otimes 3})$, then the maps of degree $0$ that send $\gr^l(V^{\otimes 3})$ to $\gr^l V$
reduce to the following:
\begin{align*}
\sum_{j\in \ZZ_{+}}(-1)^{j}\mu_{(-1-j)}(I\otimes \mu_{(-1+j)})\binom{\N_{12}}{j} &=\mu_{(-1)}(I\otimes \mu_{(-1)})\, ,\\
-\sum_{j\in \ZZ_{+}}(-1)^{j}\mu_{(-1-j)}(I\otimes \mu_{(-1+j)}) \binom{\N_{12}}{j} P_{12} &=-\mu_{(-1)}(I\otimes \mu_{(-1)}) {P_{12}}\, ,\\
\sum_{j\in \ZZ_{+}} \mu_{(-2-j)}( \mu_{(j)}\otimes I) \binom{-1+\N_{13}}{j} &=0\, .
\end{align*}
Therefore, \eqref{borcherds2} implies the commutativity 
\begin{align*}
\mu_{(-1)}(I\otimes \mu_{(-1)})-\mu_{(-1)}(I\otimes \mu_{(-1)})P_{12}=0\, .
\end{align*}

Similarly, we use \eqref{borcherds2} for $m=0$, $n=k=-1$, and note that when applied to $\gr^l(V^{\otimes 3})$, the terms of degree $0$ reduce to:
\begin{align*}
\sum_{j\in \ZZ_{+}}(-1)^{j}\mu_{(-1-j)}(I\otimes \mu_{(-1+j)})\binom{-1+\N_{12}}{j} &=\mu_{(-1)}(I\otimes \mu_{(-1)})\, ,\\
-\sum_{j\in \ZZ_{+}}(-1)^{j}\mu_{(-2-j)}(I\otimes \mu_{(j)}) \binom{-1+\N_{12}}{j} P_{12} &=0\, ,\\
\sum_{j\in \ZZ_{+}} \mu_{(-1-j)}( \mu_{(-1+j)}\otimes I) \binom{\N_{13}}{j} &= \mu_{(-1)}( \mu_{(-1)}\otimes I)\, .
\end{align*}
This implies the associativity 
\begin{equation*}
\mu_{(-1)}(I\otimes \mu_{(-1)}) =\mu_{(-1)}( \mu_{(-1)}\otimes I) \,,
\end{equation*}
thus completing the proof of the proposition.
\end{proof}

\subsection{Properties of the $\lambda$-bracket}\label{sec3.3}

In this section, we study the properties of the $\lambda$-bracket defined by \eqref{eq2.1}.

\begin{proposition}\label{propbrack}
Let\/ $V$ be a filtered logVA. 
Then $(\gr V, \partial , \{\cdot_{\lambda}\cdot\})$ satisfy the sesqui-linearity and skew-symmetry$:$
\begin{itemize}
\item[] (i) \;\,
$\{\partial a_{\lambda}b\}=-\lambda \{a_{\lambda}b\}\,;$
\smallskip

\item[] (ii) \,
$\{a_{\lambda}\partial b\}=(\lambda+\partial)\{a_{\lambda} b\}\,;$
\smallskip

\item[] (iii) \;\,
$ \{b_{\lambda}a\}=-(-1)^{p_{a,b}} \{a_{-\lambda-\partial}b\}\,.$
\end{itemize}
\end{proposition}
\begin{proof}
In this proof, we use notation \eqref{eq3.4}. $(i)$ Using parts \eqref{2-1cor5} and \eqref{2cor5} of \coref{cor}, we compute: 
\begin{align*}
[\partial a_{\lambda}b]&=\sum_{n\geq 0}\frac{\lambda^{n}}{n!} \mu_{(n)}(\partial a \otimes b) \\
&=-\sum_{n\geq 1}\frac{\lambda^{n}}{(n-1)!} \mu_{(n-1)}(a \otimes b)-\sum_{i=1}^{L} (-1)^{p_{i,a}} (\phi_{i}  a)(\psi_{i} b) \\
&=-\lambda [a_{\lambda}b]-\sum_{i=1}^{L} (-1)^{p_{i,a}} (\phi_{i} a)(\psi_{i} b) \,,
\end{align*}
and
\begin{align*}
\sum_{i=1}^{L} (-1&)^{p_{i,a}} \Bigl(\frac{1}{\lambda+\partial }\, \phi_{i} (\partial a)\Bigr) (\psi_{i} b) \\
&=\sum_{i=1}^{L} (-1)^{p_{i,a}} \Bigl(\frac{(\lambda+\partial)-\lambda}{\lambda+\partial }\, \phi_{i}  a\Bigr) (\psi_{i} b) \,\\
&=\sum_{i=1}^{L} (-1)^{p_{i,a}} \Bigl((\phi_{i} a)(\psi_{i} b) -\lambda\Bigl(\frac{1}{\lambda+\partial }\, \phi_{i}  a\Bigr) (\psi_{i} b) \Bigr) .
\end{align*}
Adding the above two identities proves $(i)$.

$(iii)$ From \coref{cor}\eqref{2cor7} and the binomial formula, we get:
\begin{equation}\label{e3}
\begin{split}
[b_{\lambda}a]&= \sum_{n\geq 0} \frac{\lambda^{n}}{n!} \mu_{(n)}(b \otimes a) \\
&= -(-1)^{p_{a,b}} \sum_{n,j\geq 0}(-1)^{n+j} \frac{\lambda^{n}}{n!} \frac{\partial^{j}}{j!} \mu_{(n+j)}(a\otimes b) \\
&=-(-1)^{p_{a,b}} \sum_{m\geq 0}\frac{(-\lambda-\partial)^{m}}{m!} \mu_{(m)}(a \otimes b) \\
&=-(-1)^{p_{a,b}} [a_{-\lambda-\partial}b]\,.
\end{split}
\end{equation}	
On the other hand, using that $\N$ is even, the symmetry of $\N$ (\prref{co3}\eqref{co3.11}), 
the commutativity of the product $\mu_{(-1)}$ (\prref{pro3.5}),
and the Leibniz rule for the action of $\partial$, we have:
\allowdisplaybreaks
\begin{align*}
\sum_{i=1}^{L} (-1&)^{p_{i,b}} \Bigl(\frac{1}{\lambda+\partial }\,\phi_{i} b\Bigr) (\psi_{i} a) \\
&= \sum_{i=1}^{L} (-1)^{p_{i,b}+p_{\psi_i a,\phi_i b}} (\psi_{i} a)\Bigl(\frac{1}{\lambda+\partial }\, \phi_{i} b\Bigr) \\
&=\sum_{i=1}^{L} (-1)^{p_{a,b}+p_{i,a}} (\phi_{i} a)\Bigl(\frac{1}{\lambda+\partial }\, \psi_{i} b\Bigr)\\
&=-(-1)^{p_{a,b}} \sum_{i=1}^{L} (-1)^{p_{i,a}} \Bigl(\Bigl(\frac{1}{\mu+\partial }\, \phi_{i}  a\Bigr)\psi_{i} b\Bigr)\Big|_{\mu=-\lambda-\partial} \,.
\end{align*}	
Adding this equation and \eqref{e3}, we obtain the skew-symmetry $(iii)$.

Finally, $(ii)$ follows from $(i)$ and $(iii)$.
\end{proof}

\begin{proposition}\label{pro3.8} 
Let\/ $V$ be a filtered logVA. Then $(\gr V,  \{\cdot_{\lambda}\cdot\}, \cdot)$ satisfy the Leibniz rule$:$
\[\{a_{\lambda}bc\}=\{a_{\lambda}b\}c+(-1)^{p_{a,b}}b\{a_{\lambda}c\}\,.\]
\end{proposition}
\begin{proof}
We apply the Borcherds identity \eqref{borcherds2} for $k=-1$, $n=0$ and $m\geq 0$.
Recall that as maps on the associated graded, $\N$ and $\mu_{(r)}$ for $r\ge0$ have degree $-1$, and $\mu_{(r)}$ for $r<0$ have degree $0$.
Hence, when acting on $\gr^l(V^{\otimes 3})$, the maps of degree $-1$ in \eqref{borcherds2} that send $\gr^l(V^{\otimes 3})$ to $\gr^{l-1}V$
reduce to the following:
\begin{align*}
\sum_{j\in \ZZ_{+}}(-1)^{j}\mu_{(m-j)}(I\otimes \mu_{(-1+j)})\binom{\N_{12}}{j} &=\mu_{(m)}(I\otimes \mu_{(-1)})\, ,\\
-\sum_{j\in \ZZ_{+}}(-1)^{j}\mu_{(-1-j)}(I\otimes \mu_{(m+j)}) \binom{\N_{12}}{j} P_{12} &=-\mu_{(-1)}(I\otimes \mu_{(m)}) {P_{12}}\, ,\\
\sum_{j\in \ZZ_{+}} \mu_{(m-1-j)}( \mu_{(j)}\otimes I) \binom{m+\N_{13}}{j} &= \mu_{(-1)}( \mu_{(m)}\otimes I)\, .
\end{align*}
This implies the identities
\[\mu_{(m)}(I\otimes \mu_{(-1)})=\mu_{(-1)} (\mu_{(m)}\otimes I)+\mu_{(-1)} (I\otimes \mu_{(m)})P_{12} \] 
for all $m\ge0$, which are equivalent to
\begin{equation}\label{e4}
[a_{\lambda}bc]=[a_{\lambda}b]c+(-1)^{p_{a,b}}b[a_{\lambda}c] \,.
\end{equation}

From the commutativity and associativity of the product (\prref{pro3.5})
and the fact that all $\psi_i$ are derivations (\prref{co3}\eqref{co3.4}), we have:
\begin{align*}
&\sum_{i=1}^{L} (-1)^{p_{i,a}} \Bigl(\frac{1}{\lambda+\partial }\, \phi_{i}a\Bigr)\psi_{i}( bc) \\
&=\sum_{i=1}^{L} (-1)^{p_{i,a}} \Bigl(\frac{1}{\lambda+\partial }\,  \phi_{i}a\Bigr)\psi_{i}( b)c
+ (-1)^{p_{i,a}+p_{i,b}} \Bigl(\frac{1}{\lambda+\partial }\,  \phi_{i}a\Bigr)b\,\psi_{i}( c)\\
&=\sum_{i=1}^{L} (-1)^{p_{i,a}} \Bigl(\frac{1}{\lambda+\partial }\,\phi_{i}a\Bigr)\psi_{i}( b)c
+(-1)^{p_{a,b}+p_{i,a}} b\Bigl(\frac{1}{\lambda+\partial }\,  \phi_{i}a\Bigr)\psi_{i}( c) \, .
\end{align*}
Adding this to \eqref{e4} completes the proof.
\end{proof}

\begin{proposition}\label{propadm}
For any filtered logVA $V$, the bracket $\{\cdot_{\lambda}\cdot\}$ on $\V=\gr V$ is admissible$:$
\begin{equation}\label{e5}
\{a_{\lambda}\{b_{\mu}c\}\}\in \iota_{\mu,\lambda} \V[[\lambda^{-1}, \mu^{-1}, (\lambda+\mu)^{-1}]][\lambda, \mu]
\end{equation}
for all $a,b,c\in\V$.
\end{proposition}
\begin{proof}
Note that \eqref{e5} implies that the $\lambda$-bracket is admissible, due to \reref{rem2.1}.
Using \eqref{eq3.4}, the Leibniz rule, and sesqui-linearity, we compute:
\begin{equation*}
\begin{split}
\{a_{\lambda}\{b_{\mu}c\}\}
&=\{a_{\lambda}[b_{\mu}c]\}+\sum_{i=1}^{L} (-1)^{p_{i,b}} \Bigl\{a_{\lambda} \Bigl(\Bigl(\frac{1}{\mu+\partial }\, \phi_{i} b\Bigr)\psi_{i} c\Bigr)\Bigr\}\\
&=\{a_{\lambda}[b_{\mu}c]\}
+\sum_{i=1}^{L} (-1)^{p_{i,b}} \Bigl(\frac{1}{\mu+(\lambda+\partial)}\{a_{\lambda}(\phi_{i}b)\}\Bigr)(\psi_{i}c) \\
&\quad+\sum_{i=1}^{L} (-1)^{p_{a,b}+p_{i,a}+p_{i,b}} \Bigl(\frac{1}{\mu+\partial}\phi_{i}b\Bigr) \{a_{\lambda}(\psi_{i}c)\}\,.
\end{split}
\end{equation*}
Recall that our convention is to expand $(\mu+\partial)^{-1}$ in negative powers of $\mu$ and non-negative powers of $\partial$ (cf.\ \eqref{e2.9e}).
Then after we replace $\partial$ with $\lambda+\partial$, we expand $(\mu+(\lambda+\partial))^{-1}$ in negative powers of $\mu$ and non-negative powers of $\lambda$ and $\partial$. Hence, by \eqref{eq2.1b} and \eqref{eq2.1a}, we have for the three terms in the right-hand side above:
\begin{align}
\label{rterm1}
&\{a_{\lambda}[b_{\mu}c]\} \in \V[[\lambda^{-1}]][\lambda, \mu]\,,\\
\label{rterm2}
&(-1)^{p_{i,b}} \Bigl(\frac{1}{\mu+(\lambda+\partial)}\{a_{\lambda}(\phi_{i}b) \}\Bigr)(\psi_{i}c) \in \iota_{\mu,\lambda} \V[[\lambda^{-1}, (\lambda+\mu)^{-1}]][\lambda]\,,\\
\label{rterm3}
&(-1)^{p_{a,b}+p_{i,a}+p_{i,b}} \Bigl(\frac{1}{\mu+\partial}\phi_{i}b\Bigr) \{a_{\lambda}(\psi_{i}c)\} \in \V[[\lambda^{-1}, \mu^{-1}]][\lambda]\,.
\end{align}
This proves \eqref{e5}.
\end{proof}

\subsection{Jacobi identity for the $\lambda$-bracket}\label{sec3.4}

\begin{proposition}\label{pro3.11} 
Let\/ $V$ be a filtered logVA. Then  $(\gr {V} , \partial ,\{ \cdot_{\lambda}\cdot\})$ is a non-local Lie conformal algebra; in particular, the bracket $\{\cdot_{\lambda}\cdot\}$ on $\V=\gr V$ satisfies the Jacobi identity
\[\{a_{\lambda}\{b_{\mu}c\}\}=\{\{a_{\lambda}b\}_{\lambda+\mu}c\}+(-1)^{p_{a,b}}\{b_{\mu}\{a_{\lambda}c\}\}\, .\]
\end{proposition}
\begin{proof} 
By Propositions \ref{propbrack} and \ref{propadm}, it only remains to prove the Jacobi identity. 
Recall that, since the bracket $\{\cdot_{\lambda}\cdot\}$ is admissible, the three terms in the Jacobi identity are identified with their images in
the space $\V_{\lambda, \mu}$; see \eqref{Vlamu} and \eqref{adm1}--\eqref{adm3}. 
We introduce the following subspaces of $\V_{\lambda, \mu}$:
\begin{align*}
\V_1&=\V[\lambda,\mu], & \V_5&=\lambda^{-1}\mu^{-1}\V[[\lambda^{-1}, \mu^{-1}]], \\
\V_2&=\lambda^{-1}\V[[\lambda^{-1}]][\mu], & \V_6&=\lambda^{-1}(\lambda+\mu)^{-1}\V[[\lambda^{-1}, (\lambda+\mu)^{-1}]], \\
\V_3&=\mu^{-1}\V[[\mu^{-1}]][\lambda], & \V_7&=\mu^{-1}(\lambda+\mu)^{-1}\V[[\mu^{-1}, (\lambda+\mu)^{-1}]], \\
\V_4&=(\lambda+\mu)^{-1}\V[[(\lambda+\mu)^{-1}]][\lambda]. & &
\end{align*}
The strategy of the proof is to show that
\begin{equation}\label{s1}
\{a_{\lambda}\{b_{\mu}c\}\}, \, \{b_{\mu}\{a_{\lambda}c\}\}, \, \{\{a_{\lambda}b\}_{\lambda+\mu}c\}\in \bigoplus_{k=1}^{7} \V_{k}\subset \V_{\lambda, \mu}\, , 
\end{equation}
and then prove the Jacobi identity in each summand.
We will denote the projections onto the summands as $\pi_{k}\colon \V_{\lambda, \mu}\rightarrow \V_{k}$.
The whole proof is rather long and is divided into several steps.

First, we observe that, for any vector space $\V$, we have 
\begin{equation*}
\V[[\lambda^{-1}]][\lambda] = \V[\lambda] \oplus \lambda^{-1}\V[[\lambda^{-1}]] \,;
\end{equation*}
the first summand corresponds to series with non-negative powers of $\la$ (i.e., polynomials), while the second to series with strictly negative powers of $\la$.
Similarly, 
\[
\V[[\lambda^{-1}, \mu^{-1}]][\lambda,\mu] = \V_1 \oplus \V_2 \oplus \V_3 \oplus \V_5 \,,
\]
and the same reasoning shows that the sum $\V_1+\dots+\V_7$ is direct. (However, it is not equal to the whole $\V_{\lambda, \mu}$.)

Next, we will determine the projections of $\{a_{\lambda}\{b_{\mu}c\}\}$ and check that it is equal to their sum.
This will be done in \leref{lemproof1} below. 
Then, in Lemmas \ref{lemproof2} and \ref{lemproof3}, we will do the same for the other two terms of the Jacobi identity;
in particular, proving \eqref{s1}.
After finding all projections, it will remain to check that they satisfy the Jacobi identity.
We do this in Lemmas \ref{lemproof4}, \ref{lemproof5}, \ref{lemproof6} below, which will complete the proof of the proposition.
\end{proof}

Here we present the sequence of lemmas needed in the proof of Proposition~\ref{pro3.11}.

\begin{lemma}\label{lemproof1}
For every $a,b,c\in\V=\gr V$, we have$:$
\allowdisplaybreaks
\begin{align*}
\pi_{1}\{a_{\lambda}\{b_{\mu}c\}\}&=[a_{\lambda}[b_{\mu}c]] \,,\\
\pi_{2}\{a_{\lambda}\{b_{\mu}c\}\}&=\sum_{i=1}^{L} (-1)^{p_{i,a}} \Bigl(\frac{1}{\lambda+\partial}\phi_{i}a\Bigr)\psi_{i}([b_{\mu}c]) \,,\\
\pi_{3}\{a_{\lambda}\{b_{\mu}c\}\}&= \sum_{i=1}^{L} (-1)^{p_{a,b}+p_{i,a}+p_{i,b}} \Bigl(\frac{1}{\mu+\partial}\phi_{i}b\Bigr) \bigl[a_{\lambda}(\psi_{i}c) \bigr] \,,\\
\pi_{4}\{a_{\lambda}\{b_{\mu}c\}\}&=\sum_{i=1}^{L} (-1)^{p_{i,b}} \Bigl(\frac{1}{(\lambda+\mu)+\partial} \bigl[a_{\lambda}(\phi_{i}b)\bigr]\Bigr)(\psi_{i}c) \,,\\
\pi_{5}\{a_{\lambda}\{b_{\mu}c\}\}&=\sum_{i,j=1}^{L} (-1)^{p_{j,a}+p_{i,b}+p_{j,b}} \Bigl(\frac{1}{\lambda+\partial}\phi_{j}a\Bigr) \Bigl(\frac{1}{\mu+\partial}\phi_{i}b\Bigr) (\psi_{j}\psi_{i}c) \,,\\
\pi_{6}\{a_{\lambda}\{b_{\mu}c\}\}
&=\sum_{i,j=1}^{L} (-1)^{p_{j,a}+p_{i,b}} \Bigl(\frac{1}{(\lambda+\mu)+\partial} \Bigl(\Bigl(\frac{1}{\lambda+\partial}\phi_{j}a\Bigr)(\psi_{j}\phi_{i}b)\Bigr)\Bigr)(\psi_{i}c) \,,\\
\pi_{7}\{a_{\lambda}\{b_{\mu}c\}\}&=0 \,,
\end{align*}
and 
$\{a_{\lambda}\{b_{\mu}c\}\}$ is equal to the sum of its projections.
\end{lemma}
\begin{proof}
In the proof of \prref{propadm}, we showed that $\{a_{\lambda}\{b_{\mu}c\}\}$ is the sum of the three terms \eqref{rterm1}, \eqref{rterm2}, \eqref{rterm3}.
We will expand these terms further by writing the $\la$-bracket as in \eqref{eq3.4}. 
In all cases, the summands will be in different spaces $\V_k$, so they will be the corresponding projections. 
%
%

The left-hand side of \eqref{rterm1} is:
\[
\{a_{\lambda}[b_{\mu}c]\}=[a_{\lambda}[b_{\mu}c]]+\sum_{i=1}^{L} (-1)^{p_{i,a}} \Bigl(\frac{1}{\lambda+\partial}\phi_{i}a\Bigr)\psi_{i}([b_{\mu}c])
\in\V_1\oplus\V_2\, ,
\]
which gives the projections $\pi_1$ and $\pi_2$ as the first and second summand, respectively.
For \eqref{rterm2}, we have:
\allowdisplaybreaks
\begin{align*}
\sum_{i=1}^{L} &\, (-1)^{p_{i,b}} \Bigl(\frac{1}{\mu+(\lambda+\partial)} \{a_{\lambda}(\phi_{i}b) \}\Bigr)(\psi_{i}c) \\
&=\iota_{\mu,\la} \sum_{i=1}^{L} (-1)^{p_{i,b}} \Bigl(\frac{1}{\lambda+\mu+\partial}[a_{\lambda}(\phi_{i}b)]\Bigr)(\psi_{i}c)\\
&+\iota_{\mu,\la} \sum_{i,j=1}^{L} (-1)^{p_{j,a}+p_{i,b}} \Bigl(\frac{1}{\lambda+\mu+\partial} \Bigl(\Bigl(\frac{1}{\lambda+\partial}\phi_{j}a\Bigr)(\psi_{j}\phi_{i}b)\Bigr)\Bigr)(\psi_{i}c) \\
& \in \iota_{\mu,\la} \V_4\oplus\iota_{\mu,\la} \V_6 \,,
\end{align*}
giving the projections $\pi_4$ and $\pi_6$.
Finally, for \eqref{rterm3}, we have:
\begin{align*}
\sum_{i=1}^{L} &\, (-1)^{p_{a,b}+p_{i,a}+p_{i,b}} \Bigl(\frac{1}{\mu+\partial}\phi_{i}b\Bigr) \{a_{\lambda}(\psi_{i}c)\} \\
&=\sum_{i=1}^{L} (-1)^{p_{a,b}+p_{i,a}+p_{i,b}} \Bigl(\frac{1}{\mu+\partial}\phi_{i}b\Bigr) [a_{\lambda}(\psi_{i}c)]\\
&+\sum_{i,j=1}^{L} (-1)^{p_{a,b}+p_{i,a}+p_{i,b}+p_{j,a}} \Bigl(\frac{1}{\mu+\partial}\phi_{i}b\Bigr) \Bigl(\frac{1}{\lambda+\partial}\phi_{j}a\Bigr)(\psi_{j}\psi_{i}c) \\
&\in\V_3\oplus\V_5 \,,
\end{align*}
which gives $\pi_3$ and $\pi_5$ after we switch the first two factors in the last sum using the commutativity of the product.
\end{proof}

\begin{lemma}\label{lemproof2}
For every $a,b,c\in\V=\gr V$, we have$:$
\allowdisplaybreaks
\begin{align*}
\pi_{1}\{b_{\mu}\{a_{\lambda}c\}\} &=[b_{\mu}[a_{\lambda}c]] \\
&-\sum_{i=1}^{L} (-1)^{p_{a,b}+p_{i,a}+p_{i,b}} \frac{\bigl[(\phi_{i}a)_{-\mu-\partial} b\bigr]-\bigl[(\phi_{i}a)_{\la} b\bigr]}{\lambda+\mu+\partial} (\psi_{i}c) \,,\\
\pi_{2}\{b_{\mu}\{a_{\lambda}c\}\} &= \sum_{i=1}^{L} (-1)^{p_{a,b}+p_{i,a}+p_{i,b}} \Bigl(\frac{1}{\lambda+\partial}\phi_{i}a\Bigr) \bigl[b_{\mu}(\psi_{i}c) \bigr] \,,\\
\pi_{3}\{b_{\mu}\{a_{\lambda}c\}\} &=\sum_{i=1}^{L} (-1)^{p_{i,b}} \Bigl(\frac{1}{\mu+\partial}\phi_{i}b\Bigr)\psi_{i}([a_{\lambda}c]) \,,\\
\pi_{4}\{b_{\mu}\{a_{\lambda}c\}\} &= -\sum_{i=1}^{L} (-1)^{p_{a,b}+p_{i,a}+p_{i,b}} \Bigl(\frac{1}{(\lambda+\mu)+\partial} \bigl[(\phi_{i}a)_{\lambda}b\bigr]\Bigr)(\psi_{i}c) \,,\\
\pi_{5}\{b_{\mu}\{a_{\lambda}c\}\} &= \sum_{i,j=1}^{L} (-1)^{p_{i,a}+p_{j,a}+p_{j,b}} \Bigl(\frac{1}{\mu+\partial}\phi_{j}b\Bigr) \Bigl(\frac{1}{\lambda+\partial}\phi_{i}a\Bigr) (\psi_{j}\psi_{i}c) \,,\\
\pi_{6}\{b_{\mu}\{a_{\lambda}c\}\} &=0 \,,\\
\pi_{7}\{b_{\mu}\{a_{\lambda}c\}\} &=\sum_{i,j=1}^{L} (-1)^{p_{i,a}+p_{j,b}} \Bigl(\frac{1}{(\lambda+\mu)+\partial} \Bigl(\Bigl(\frac{1}{\mu+\partial}\phi_{j}b\Bigr)(\psi_{j}\phi_{i}a)\Bigr)\Bigr)(\psi_{i}c)  \,,
\end{align*}
and\/ $\{b_{\mu}\{a_{\lambda}c\}\}$ is equal to the sum of its projections.
\end{lemma}
\begin{proof}
We switch $a \leftrightarrow b$, $\la\leftrightarrow\mu$ in the results of \leref{lemproof2}, and note that under $\la\leftrightarrow\mu$ we have
\[
\V_1\leftrightarrow\V_1, \qquad \V_2\leftrightarrow\V_3, \qquad \V_5\leftrightarrow\V_5, \qquad \V_6\leftrightarrow\V_7.
\]
However, under $\la\leftrightarrow\mu$ we have
\[
\V_4 
\leftrightarrow (\lambda+\mu)^{-1}\V[[(\lambda+\mu)^{-1}]][\mu] \subset \V_1 \oplus \V_4 \,,
\]
where the last inclusion is obtained by the substitution $\mu=(\la+\mu)-\la$.

This proves the formulas for $\pi_{k}\{b_{\mu}\{a_{\lambda}c\}\}$ with $k=2,3,5,6,7$.
To find the projections for $k=1,4$, we have to decompose
\[
\sum_{i=1}^{L} (-1)^{p_{i,a}} \Bigl(\frac{1}{(\lambda+\mu)+\partial} [b_{\mu}(\phi_{i}a)]\Bigr)(\psi_{i}c) 
\]
according to $\V_1 \oplus \V_4$, and the terms in $\V_1$ will be added to $[b_{\mu}[a_{\lambda}c]]$ to give $\pi_{1}\{b_{\mu}\{a_{\lambda}c\}\}$,
while the terms in $\V_4$ will give $\pi_{4}\{b_{\mu}\{a_{\lambda}c\}\}$.

In order to do that, it will be convenient to first apply the skew-symmetry \eqref{e3}:
\[
[b_{\mu}(\phi_{i}a)] = -(-1)^{p_{a,b}+p_{i,b}} [(\phi_{i}a)_{-\mu-\partial} b] \,.
\]
Then we write
\[
\frac{1}{(\lambda+\mu)+\partial} [(\phi_{i}a)_{-\mu-\partial} b] = 
\frac{[(\phi_{i}a)_{-\mu-\partial} b]-[(\phi_{i}a)_{\la} b]}{\lambda+\mu+\partial}
+ \frac{1}{(\lambda+\mu)+\partial} [(\phi_{i}a)_{\la} b] \,.
\]
The first summand in the right-hand side is in fact a polynomial in $\la$, $\mu$ and $\partial$, because
$\lambda+\mu+\partial=\la - (-\mu-\partial)$. Hence, the first summand is in $\V_1$, while the second is in $\V_4$.
\end{proof}

\begin{lemma}\label{lemproof3}
For every $a,b,c\in\V=\gr V$, we have$:$
\allowdisplaybreaks
\begin{align*}
\pi_{1}&\{\{a_{\lambda}b\}_{\lambda+\mu}c\} = [[a_{\lambda} b]_{\lambda+\mu}c] \\
&+\sum_{i=1}^{L} (-1)^{p_{a,b}+p_{a,c}+p_{i,b}+p_{i,c}+p_i} 
\frac{ \bigl[ (\psi_{i} b)_{\lambda+\mu+\partial}c \bigr]_{\rightarrow} - \bigl[(\psi_{i} b)_{\mu}c \bigr] }{\la+\partial} (\phi_{i} a) \\
&-\sum_{i=1}^{L} (-1)^{p_{b,c}+p_{i,a}+p_{i,c}} \frac{ \bigl[ (\phi_{i} a)_{\lambda+\mu+\partial}c \bigr]_{\rightarrow} - \bigl[ (\phi_{i} a)_{\lambda}c \bigr] }{\mu+\partial} (\psi_{i} b) \,, \\
\pi_2&\{\{a_{\lambda}b\}_{\lambda+\mu}c\} =\sum_{i=1}^{L} (-1)^{p_{i,a}} \Bigl(\frac{1}{\lambda+\partial }\phi_{i} a\Bigr) \bigl[ (\psi_{i} b)_{\mu}c \bigr] \,, \\
\pi_3&\{\{a_{\lambda}b\}_{\lambda+\mu}c\} =-\sum_{i=1}^{L} (-1)^{p_{b,c}+p_{i,a}+p_{i,c}} \bigl[(\phi_{i} a)_{\lambda}c\bigr] \Bigl(\frac{1}{\mu+\partial }\psi_{i} b\Bigr) \,,\\
\pi_4&\{\{a_{\lambda}b\}_{\lambda+\mu}c\} =\sum_{i=1}^{L} (-1)^{p_{i,a}+p_{i,b}} \Bigl(\frac{1}{(\lambda+\mu)+\partial}\phi_{i}([a_{\lambda}b])\Bigr) (\psi_{i}c) \,,\\
\pi_5&\{\{a_{\lambda}b\}_{\lambda+\mu}c\} = 0 \,,\\
\pi_6&\{\{a_{\lambda}b\}_{\lambda+\mu}c\} \\
&=\sum_{i,j=1}^{L} (-1)^{p_{i,a}+p_{j,b}+p_{i,j}} \Bigl(\frac{1}{(\lambda+\mu)+\partial} \Bigl(\Bigl(\frac{1}{\lambda+\partial }\phi_{i} a\Bigr) (\phi_{j}\psi_{i}b)\Bigr)\Bigr) (\psi_{j}c) \,,\\
\pi_7&\{\{a_{\lambda}b\}_{\lambda+\mu}c\} \\
&=-\sum_{i,j=1}^{L} (-1)^{p_{i,a}+p_{j,a}+p_{j,b}} \Bigl(\frac{1}{(\lambda+\mu)+\partial}\Bigl((\phi_{j}\phi_{i}a) \Bigl(\frac{1}{\mu+\partial }\psi_{i} b\Bigr)\Bigr)\Bigr)(\psi_{j}c) \,,
\end{align*}
where we use the notation \eqref{earr}.
Moreover, $\{\{a_{\lambda}b\}_{\lambda+\mu}c\}$ is equal to the sum of its projections.
\end{lemma}
\begin{proof}
It is possible to derive these projections from \leref{lemproof1}, by using the skew-symmetry to write
\[
\{\{a_{\lambda}b\}_{\lambda+\mu}c\} = -(-1)^{p_{a,c}+p_{b,c}} \{c_{-\lambda-\mu-\partial} \{a_{\lambda}b\}\} \,,
\]
and then making the appropriate substitutions in the results of \leref{lemproof1}, similarly to the proof of \leref{lemproof2}.
Instead, we present here a more direct proof, which is done by an explicit calculation of $\{\{a_{\lambda}b\}_{\lambda+\mu}c\}$.

We start by writing $\{a_{\lambda}b\}$ as in \eqref{eq3.4} as a sum of $[a_{\lambda}b]$ and another term involving the commutative associative product on $\V$.
Then, when taking the $(\lambda+\mu)$-bracket of a product with $c$, we apply the right Leibniz rule \eqref{lm2.6} (which follows from the Leibniz rule and skew-symmetry).
We get:
\allowdisplaybreaks
\begin{equation}\label{e7}
\begin{split}
\{&\{a_{\lambda}b\}_{\lambda+\mu}c\}
=\{[a_{\lambda} b]_{\lambda+\mu}c\} \\
&+\sum_{i=1}^{L} (-1)^{p_{i,a}+p_{b,c}+p_{i,c}} \Bigl\{\Bigl(\frac{1}{\lambda+\partial }\, \phi_{i} a\Bigr)_{\lambda+\mu+\partial}c\Bigr\}_{\rightarrow }(\psi_{i} b)\\
&+\sum_{i=1}^{L} (-1)^{p_{a,b}+p_{a,c}+p_{i,b}+p_{i,c}+p_i} \Bigl\{(\psi_{i} b)_{\lambda+\mu+\partial}c\Bigr\}_{\rightarrow } \Bigl(\frac{1}{\lambda+\partial }\, \phi_{i} a\Bigr).
\end{split}
\end{equation}
The parity in the last term was obtained from the following calculation:
\begin{align*}
p_{i,a}+p_{\phi_ia,c}+p_{\phi_ia,\psi_ib} &= p_{i,a}+(p_i+p_a)p_c+(p_i+p_a)(p_i+p_b) \\
&= p_{a,b}+p_{a,c}+p_{i,b}+p_{i,c}+p_i \mod 2\ZZ.
\end{align*}
The meaning of the arrows $\to$ in the right-hand side of \eqref{e7} is that after we compute the $(\la+\mu+\partial)$-bra\-cket, $\partial$ is applied to the terms on the right of the arrow. However, this is not done for the $\partial$ from $1/(\lambda+\partial)$ in the second summand.

Now we evaluate the $\{\,\}$-bracket for each term in the right-hand side of \eqref{e7}, using again \eqref{eq3.4}.
The first term gives
\begin{align*}
\{[a_{\lambda} b&]_{\lambda+\mu}c\} = [[a_{\lambda} b]_{\lambda+\mu}c] \\
&+\sum_{i=1}^{L} (-1)^{p_{i,a}+p_{i,b}}\Bigl(\frac{1}{(\lambda+\mu)+\partial}\phi_{i}([a_{\lambda}b])\Bigr)(\psi_{i}c)
\in\V_1\oplus\V_4 \,.
\end{align*}
The above two summands lie in $\V_1$ and $\V_4$, respectively, and contribute to the projections $\pi_1$ and $\pi_4$
of $\{\{a_{\lambda}b\}_{\lambda+\mu}c\}$.

For the second term of \eqref{e7}, we first compute
\begin{equation}\label{e9}
\begin{split}
\Bigl\{&\Bigl(\frac{1}{\lambda+\partial}\, \phi_{i} a\Bigr)_{\nu}c\Bigr\}
= \frac{1}{\lambda-\nu} \bigl\{(\phi_{i} a)_\nu c\bigr\} \\
&= \frac{1}{\lambda-\nu} \bigl[(\phi_{i} a)_\nu c\bigr] 
+\sum_{j=1}^{L} (-1)^{p_{j,a}+p_{i,j}} \frac{1}{\lambda-\nu} \Bigl(\Bigl(\frac{1}{\nu+\partial} (\phi_{j}\phi_{i}a)\Bigr) (\psi_{j}c)\Bigr),
\end{split}
\end{equation}
where we used sesqui-linearity and \eqref{eq3.4}. To get the second term of \eqref{e7}, we need to multiply this expression 
by $\psi_{i} b$ on the right and replace $\nu$ with $\la+\mu+\partial$ where $\partial$ is applied to $\psi_{i} b$. 
Hence, the factor $\frac1{\lambda-\nu}$ produces $-\frac1{\mu+\partial}(\psi_{i} b)$. 
Then the first summand in the right-hand side of \eqref{e9} gives
\[
-\bigl[ (\phi_{i} a)_{\lambda+\mu+\partial}c \bigr]_{\rightarrow} \Bigl(\frac{1}{\mu+\partial }\psi_{i} b\Bigr).
\]
In the second summand, the factor $\frac1{\nu+\partial}$ has $\partial$ that is applied to $\phi_{j}\phi_{i}a$,
while after replacing $\nu$ with $\la+\mu+\partial$ its $\partial$ is applied to $\psi_{i} b$. 
Hence, by the Leibniz rule for $\partial$, the overall result is that $\frac1{\nu+\partial}$ becomes
$\frac1{\la+\mu+\partial}$ with $\partial$ applied to the product $(\phi_{j}\phi_{i}a)(\psi_{i} b)$.
In order to write the two factors of this product next to each other, we switch $\psi_{j} c$ with $\psi_{i} b$
using commutativity:
\[
(\psi_{j} c)(\psi_{i} b) = (-1)^{p_{b,c}+p_{j,b}+p_{i,c}+p_{i,j}} (\psi_{i} b)(\psi_{j} c) \,.
\]
Thus, the second summand in the right-hand side of \eqref{e9} gives
\[
-\sum_{j=1}^{L} (-1)^{p_{j,a}+p_{b,c}+p_{j,b}+p_{i,c}} 
\Bigl(\frac{1}{\lambda+\mu+\partial}\Bigl((\phi_{j}\phi_{i}a) \frac{1}{\mu+\partial }\psi_{i} b\Bigr)\Bigr)(\psi_{j}c) \,.
\]
Adding the two summands of \eqref{e9} together, we obtain that the second term of \eqref{e7} is
\allowdisplaybreaks
\begin{align*}
\sum_{i=1}^{L} & (-1)^{p_{i,a}+p_{b,c}+p_{i,c}} \Bigl\{\Bigl(\frac{1}{\lambda+\partial }\, \phi_{i} a\Bigr)_{\lambda+\mu+\partial}c\Bigr\}_{\rightarrow }(\psi_{i} b)\\
&= -\sum_{i=1}^{L} (-1)^{p_{i,a}+p_{b,c}+p_{i,c}} \bigl[ (\phi_{i} a)_{\lambda+\mu+\partial}c \bigr]_{\rightarrow} \Bigl(\frac{1}{\mu+\partial }\psi_{i} b\Bigr) \\
&- \sum_{i,j=1}^{L} (-1)^{p_{i,a}+p_{j,a}+p_{j,b}} 
\Bigl(\frac{1}{\lambda+\mu+\partial}\Bigl((\phi_{j}\phi_{i}a) \frac{1}{\mu+\partial }\psi_{i} b\Bigr)\Bigr)(\psi_{j}c) \,.
\end{align*}
The second sum above lies in $\V_7$ and thus contributes to $\pi_7\{\{a_{\lambda}b\}_{\lambda+\mu}c\}$ (in fact, it will be equal to it).
The first sum is in $\mu^{-1} \V[[\mu^{-1}]][\lambda,\mu]\subset \V_1\oplus\V_3$ and decomposes according to:
\begin{align*}
\bigl[ &(\phi_{i} a)_{\lambda+\mu+\partial}c \bigr]_{\rightarrow} \Bigl(\frac{1}{\mu+\partial }\psi_{i} b\Bigr) \\
&= \frac{ \bigl[ (\phi_{i} a)_{\lambda+\mu+\partial}c \bigr]_{\rightarrow} - \bigl[ (\phi_{i} a)_{\lambda}c \bigr] }{\mu+\partial} (\psi_{i} b)
+ \bigl[ (\phi_{i} a)_{\lambda}c \bigr] \Bigl(\frac{1}{\mu+\partial }\psi_{i} b\Bigr).
\end{align*}
These two terms contribute to $\pi_1\{\{a_{\lambda}b\}_{\lambda+\mu}c\}$ and $\pi_3\{\{a_{\lambda}b\}_{\lambda+\mu}c\}$, respectively.

Finally, we determine the third term in the right-hand side of \eqref{e7} in a similar (but easier) manner. We obtain:
\begin{align*}
\Big\{ (\psi_{i} b&)_{\lambda+\mu+\partial}c\Big\}_{\rightarrow } \Bigl(\frac{1}{\lambda+\partial }\, \phi_{i} a\Bigr)
= \bigl[(\psi_{i} b)_{\lambda+\mu+\partial}c \bigr]_{\rightarrow} \Bigl(\frac{1}{\lambda+\partial }\phi_{i} a\Bigr)\\
&+\sum_{j=1}^{L} (-1)^{s_{i,j}} 
\Bigl(\frac{1}{\lambda+\mu+\partial}\Bigl(\Bigl( \frac{1}{\lambda+\partial }\phi_{i} a\Bigr)\phi_{j}\psi_{i}b\Bigr)\Bigr)(\psi_{j}c),
\end{align*}
where
\[
s_{i,j} = p_{a,b}+p_{a,c}+p_{i,a}+p_{i,b}+p_{i,c}+p_{i,j}+p_i+p_{j,b} \,.
\]
The second summand above lies in $\V_6$ and thus contributes to $\pi_6\{\{a_{\lambda}b\}_{\lambda+\mu}c\}$.
The first summand is in $\la^{-1} \V[[\la^{-1}]][\lambda,\mu]\subset \V_1\oplus\V_2$ and decomposes according to:
\begin{align*}
\bigl[ &(\psi_{i} b)_{\lambda+\mu+\partial}c \bigr]_{\rightarrow} \Bigl(\frac{1}{\lambda+\partial }\phi_{i} a\Bigr) \\
&= \frac{ \bigl[ (\psi_{i} b)_{\lambda+\mu+\partial}c \bigr]_{\rightarrow} - \bigl[(\psi_{i} b)_{\mu}c \bigr] }{\la+\partial} (\phi_{i} a)
+ \bigl[(\psi_{i} b)_{\mu}c \bigr] \Bigl(\frac{1}{\lambda+\partial }\phi_{i} a\Bigr).
\end{align*}
These two terms contribute to the projections $\pi_1$ and $\pi_2$ of $\{\{a_{\lambda}b\}_{\lambda+\mu}c\}$, respectively,
and we improve the answer for the latter by switching the two factors.
This completes the proof of \leref{lemproof3}.
\end{proof}

Now that we have proven \eqref{s1}, we will check the Jacobi identity in each summand $\V_1,\dots,\V_7$.
In the proof of the Jacobi identity in $\V_1$, we will need a lemma about binomial coefficients, in which we use the following standard notation.
For a polynomial
\[f(x)=\sum_{n=0}^{N}f_{n}x^{n}\in \CC[x]\,, \] 
we will write 
\[f(x)=f_{0}+O(x)\, ,\qquad
f(x)=f_{0}+f_{1}x+O(x^{2})\, ,\]
where $O(x)$ is given by terms of order $1$ or higher and $O(x^{2})$ is given by terms of order $2$ or higher.

\begin{lemma}\label{2lm} 
The binomial coefficients satisfy the following identities$:$
\medskip
\begin{enumerate}
\item\label{lm1} 
\;$\displaystyle\binom{x}{j}=\frac{(-1)^{j-1}}{j}x+O(x^{2})$  \; for \; $0< j\,;$
\medskip

\item\label{lm2} 
\;$\displaystyle\binom{m+x}{j}=\binom{m}{j}+O(x)$ \; for \; $0\leq j\leq m\,;$
\medskip

\item\label{lm3} 
\;$\displaystyle\binom{m+x}{j}=\frac{(-1)^{j+m+1}}{j\binom{j-1}{m}}x+O(x^{2})$ \; for \; $0\leq m< j\,;$
\medskip

\item\label{lm4} 
\;$\displaystyle\sum_{l=0}^{m}\frac{(-1)^{l}}{l+n+1}\binom{m}{l}=\frac{1}{(n+m+1)\binom{n+m}{m}}\,.$
\end{enumerate}
\end{lemma}
\begin{proof} $(1)$ is a special case of $(3)$. Identity $(2)$ follows by evaluating the left-hand side at $x=0$. To prove $(3)$, note that for $j>m\geq 0$ we have:
\begin{equation*}
\begin{split}
\binom{m+x}{j}&=\frac{(x+m)\cdots (x+1)x(x-1) \cdots (x+m-j+1)}{j!}\\
&=\frac{m!\,x(-1) \cdots (m-j+1)}{j!}+O(x^{2})\\
&=\frac{(-1)^{j-m-1}}{j\binom{j-1}{m}}x+O(x^{2})\, .
\end{split}
\end{equation*}
$(4)$ follows from the identity
\begin{align*}
&\sum_{l=0}^{m}\frac{(-1)^{l}}{l+n+1}\binom{m}{l}=\int_{0}^{1}\sum_{l=0}^{m}(-1)^{l}\binom{m}{l}x^{l}x^{n}dx\\
&\qquad =\int_{0}^{1}x^{n}(1-x)^{m}dx=\frac{n!\,m!}{(n+m+1)!}\, .
\end{align*}
This completes the proof of the lemma.
\end{proof}

\begin{lemma}\label{lemproof4}
The $\pi_1$-projection of the Jacobi identity holds, i.e.,
\begin{equation}\label{eq3.24}
\pi_{1}\{a_{\lambda}\{b_{\mu}c\}\}=\pi_{1}\{\{a_{\lambda}b\}_{\lambda+\mu}c\}
+(-1)^{p_{a,b}} \pi_{1}\{b_{\mu}\{a_{\lambda}c\}\}
\end{equation}
for every\/ $a,b,c\in\V=\gr V$.
\end{lemma}
\begin{proof}
By Lemmas \ref{lemproof1}--\ref{lemproof3}, identity \eqref{eq3.24} is equivalent to:
\begin{equation}\label{e11}
\begin{split}
[a_{\lambda}&[b_{\mu}c]] = (-1)^{p_{a,b}} [b_{\mu}[a_{\lambda}c]] + [[a_{\lambda} b]_{\lambda+\mu}c] \\
&+\sum_{i=1}^{L} (-1)^{p_{a,b}+p_{a,c}+p_{i,b}+p_{i,c}+p_i} 
\frac{ \bigl[ (\psi_{i} b)_{\lambda+\mu+\partial}c \bigr]_{\rightarrow} - \bigl[(\psi_{i} b)_{\mu}c \bigr] }{\la+\partial} (\phi_{i} a) \\
&-\sum_{i=1}^{L} (-1)^{p_{b,c}+p_{i,a}+p_{i,c}} 
\frac{ \bigl[ (\phi_{i} a)_{\lambda+\mu+\partial}c \bigr]_{\rightarrow} - \bigl[ (\phi_{i} a)_{\lambda}c \bigr] }{\mu+\partial} (\psi_{i} b) \\
&-\sum_{i=1}^{L} (-1)^{p_{i,a}+p_{i,b}} \frac{\bigl[(\phi_{i}a)_{-\mu-\partial} b\bigr]-\bigl[(\phi_{i}a)_{\la} b\bigr]}{\lambda+\mu+\partial} (\psi_{i}c) \,,
\end{split}
\end{equation}
Let us extract the coefficient of $\lambda^{(m)}\mu^{(k)}$ from each term of \eqref{e11}, for fixed $m,k\in\ZZ_+$, where we use the divided-powers notation
$\lambda^{(m)}:=\lambda^{m}/m!$.

By definition \eqref{eq2.1b}, the term
\[
[a_{\lambda}[b_{\mu}c]] = \sum_{m,k\in\ZZ_+} \la^{(m)}\mu^{(k)} a_{(m+\N)} (b_{(k+\N)}c)
\]
gives the coefficient
\[
a_{(m+\N)} (b_{(k+\N)}c) = \mu_{(m)} (I \otimes \mu_{(k)}) (a \otimes b \otimes c) \,.
\]
Similarly, the coefficient in $(-1)^{p_{a,b}} [b_{\mu}[a_{\lambda}c]]$ is
\[
(-1)^{p_{a,b}} b_{(k+\N)} (a_{(m+\N)} c) = \mu_{(k)} (I \otimes \mu_{(m)}) P_{12} (a \otimes b \otimes c) \,.
\]
We will express all other coefficients in terms of compositions of products $\mu_{(n)}$, the action of $\N$ defined by \eqref{logf-nact}, 
and permutations applied to $a \otimes b \otimes c$;
in this way deducing \eqref{e11} as a special case of the Borcherds identity \eqref{borcherds2}.

We will utilize the binomial formula
\[
(\la+\mu)^{(n)} = \sum_{l\in\ZZ_+} \la^{(n-l)} \mu^{(l)} \,,
\]
where we set $\la^{(m)} := 0$ for $m<0$. Then
\begin{align*}
[[a_{\lambda} b]_{\lambda+\mu}c] &= \sum_{j,n\in\ZZ_+} \la^{(j)} (\la+\mu)^{(n)} (a_{(j+\N)} b)_{(n+\N)} c \\
&= \sum_{j,l,n\in\ZZ_+} \la^{(j)} \la^{(n-l)} \mu^{(l)} (a_{(j+\N)} b)_{(n+\N)} c
\end{align*}
has coefficient in front of $\lambda^{(m)}\mu^{(k)}$ equal to:
\[
\sum_{j\in\ZZ_+} \binom{m}{j} (a_{(j+\N)} b)_{(m+k-j+\N)} c 
= \sum_{j=0}^m \binom{m}{j} \mu_{(m+k-j)} (\mu_{(j)} \otimes I) (a \otimes b \otimes c) \,.
\]

Using again the binomial formula, we have
\[
\bigl[ (\psi_{i} b)_{\lambda+\mu+\partial}c \bigr]_{\rightarrow} 
= \sum_{j,n\in\ZZ_+} \mu^{(n-j)} \bigl((\psi_{i} b)_{(n+\N)} c\bigr) (\la+\partial)^{(j)} \,.
\]
From here, we find
\begin{align*}
&\frac{ \bigl[ (\psi_{i} b)_{\lambda+\mu+\partial}c \bigr]_{\rightarrow} - \bigl[(\psi_{i} b)_{\mu}c \bigr] }{\la+\partial} (\phi_{i} a) \\
&\qquad= \sum_{n\in\ZZ_+} \sum_{j\ge1} \frac1j \mu^{(n-j)} \bigl((\psi_{i} b)_{(n+\N)} c\bigr) \bigl((\la+\partial)^{(j-1)} (\phi_{i} a)\bigr) \\
&\qquad= \sum_{r,n\in\ZZ_+} \sum_{j\ge1} \frac1j \la^{(r)} \mu^{(n-j)} \bigl((\psi_{i} b)_{(n+\N)} c\bigr) \bigl(\partial^{(j-1-r)} (\phi_{i} a)\bigr) \,.
\end{align*}
Hence, after using commutativity of the product, the coefficient of $\lambda^{(m)}\mu^{(k)}$ in the first sum in \eqref{e11} is:
\begin{align*}
\sum_{i=1}^{L} & \sum_{j>m} (-1)^{p_{i,a}}
\frac1j \bigl(\partial^{(j-1-m)} (\phi_{i} a)\bigr) \bigl((\psi_{i} b)_{(k+j+\N)} c\bigr) \\
&= \sum_{j>m} \frac1j
\mu_{(-1)} (I \otimes \mu_{(k+j)}) \N_{12} \bigl(\partial^{(j-1-m)} a \otimes b \otimes c\bigr) \\
&= \sum_{j>m} \frac1j
\mu_{(m-j)} (I \otimes \mu_{(k+j)}) \N_{12} (a \otimes b \otimes c) \,,
\end{align*}
where we also used \coref{cor}\eqref{2cor5} in the last equality.

In the same way, we have
\begin{align*}
&\frac{ \bigl[ (\phi_{i} a)_{\lambda+\mu+\partial}c \bigr]_{\rightarrow} - \bigl[ (\phi_{i} a)_{\lambda}c \bigr] }{\mu+\partial} (\psi_{i} b) \\
&\qquad= \sum_{r,n\in\ZZ_+} \sum_{j\ge1} \frac1j \la^{(n-j)} \mu^{(r)} \bigl((\phi_{i} a)_{(n+\N)} c\bigr) \bigl(\partial^{(j-1-r)} (\psi_{i} b)\bigr) \,,
\end{align*}
and the coefficient of $\lambda^{(m)}\mu^{(k)}$ in the second sum in \eqref{e11} is:
\begin{align*}
-&\sum_{i=1}^{L} \sum_{j>k} (-1)^{p_{a,b}+p_{i,b}+p_{i}} 
\frac1j \bigl(\partial^{(j-1-k)} (\psi_{i} b)\bigr) \bigl((\phi_{i} a)_{(m+j+\N)} c\bigr) \\
&= -\sum_{j>k} \frac1j \mu_{(k-j)} (I \otimes \mu_{(m+j)}) \N_{12} P_{12} (a \otimes b \otimes c) \,,
\end{align*}
where we also used the symmetry of $\N$ (\prref{co3}\eqref{co3.11}).

Finally, again from the binomial formula, 
we have
\begin{align*}
&\frac{\bigl[(\phi_{i}a)_{-\mu-\partial} b\bigr]-\bigl[(\phi_{i}a)_{\la} b\bigr]}{\lambda+\mu+\partial} \\
&\quad= \sum_{j\in\ZZ_+} \sum_{n\ge1} \frac{(-1)^n}{n} \la^{(j-n)} (\la+\mu+\partial)^{(n-1)} \bigl((\phi_{i} a)_{(j+\N)} b\bigr) \\
&\quad= \sum_{j,l,r\in\ZZ_+} \sum_{n\ge1} \frac{(-1)^n}{n} \la^{(j-n)} \la^{(l)} \mu^{(r)} \partial^{(n-1-l-r)} \bigl((\phi_{i} a)_{(j+\N)} b\bigr) \,.
\end{align*}
The coefficient of $\lambda^{(m)}\mu^{(k)}$ in this expression is
\[
\sum_{j>m+k} \sum_{l=0}^m \frac{(-1)^{l+j-m}}{l+j-m} \binom{m}{l} \, \partial^{(j-m-k-1)} \bigl((\phi_{i}a)_{(j+\N)}b \bigr) \,,
\]
which thanks to \leref{2lm}\eqref{lm4} can be simplified to
\[
\sum_{j>m+k} \frac{(-1)^{j+m}}{j\binom{j-1}{m}} \, \partial^{(j-m-k-1)} \bigl((\phi_{i}a)_{(j+\N)}b \bigr) \,.
\]
Hence, the coefficient of $\lambda^{(m)}\mu^{(k)}$ in the third sum in \eqref{e11} is:
\begin{align*}
\sum_{i=1}^{L} & \sum_{j>m+k}
\frac{(-1)^{p_{i,a}+p_{i,b}+j+m+1}}{j\binom{j-1}{m}} \Bigl(\partial^{(j-m-k-1)} \bigl((\phi_{i}a)_{(j+\N)}b \bigr)\Bigr) (\psi_{i}c) \\
&= \sum_{j>m+k} \frac{(-1)^{j+m+1}}{j\binom{j-1}{m}} \mu_{(m+k-j)} (\mu_{(j)} \otimes I) \N_{13} (a \otimes b \otimes c) \,.
\end{align*}

Combining the above results, we see that identities \eqref{e11} for all $a,b,c\in\V$ are equivalent to the following collection of identities 
$(m,k\in\ZZ_+)$:
\begin{equation}\label{e12}
\begin{split}
&\mu_{(m)} (I \otimes \mu_{(k)}) = \mu_{(k)} (I \otimes \mu_{(m)}) P_{12} + \sum_{j=0}^m \binom{m}{j} \mu_{(m+k-j)} (\mu_{(j)} \otimes I) \\
&\quad+\sum_{j>m} \frac1j \mu_{(m-j)} (I \otimes \mu_{(k+j)}) \N_{12}
-\sum_{j>k} \frac1j \mu_{(k-j)} (I \otimes \mu_{(m+j)}) \N_{12} P_{12} \\
&\quad+\sum_{j>m+k} \frac{(-1)^{j+m+1}}{j\binom{j-1}{m}} \mu_{(m+k-j)} (\mu_{(j)} \otimes I) \N_{13} \,.
\end{split}
\end{equation}

Now we derive \eqref{e12} from the Borcherds identity \eqref{borcherds2} for $n=0$ and $m,k\geq 0$.
Recall that as maps on the associated graded, $\N$ and $\mu_{(r)}$ for $r\ge0$ have degree $-1$, while $\mu_{(r)}$ for $r<0$ have degree $0$.
Hence, all terms of \eqref{borcherds2} induce corresponding maps on $(\gr V)^{\otimes 3}$.
Using \leref{2lm}, we determine the maps of degree $-2$ that send $\gr^l(V^{\otimes 3})$ to $\gr^{l-2}V$
as follows:
\begin{align*}
\sum_{j\in \ZZ_{+}} &(-1)^{j}\mu_{(m-j)}(I\otimes \mu_{(k+j)})\binom{\N_{12}}{j} \\
&=\mu_{(m)}(I\otimes \mu_{(k)})
-\sum_{j>m}\frac{1}{j}\mu_{(m-j)}(I\otimes \mu_{(k+j)})\N_{12}\, , \\
-\sum_{j\in \ZZ_{+}} &(-1)^{j}\mu_{(k-j)}(I\otimes \mu_{(m+j)}) \binom{\N_{12}}{j} P_{12} \\
&=-\mu_{(k)}(I\otimes \mu_{(m)}) P_{12}
+\sum_{j>k}\frac{1}{j}\mu_{(k-j)}(I\otimes \mu_{(m+j)}) \N_{12} P_{12} \, ,
\end{align*}
and
\begin{align*}
\sum_{j\in \ZZ_{+}} &\mu_{(m+k-j)}( \mu_{(j)}\otimes I) \binom{m+\N_{13}}{j}=\sum_{j=0}^{m} \binom{m}{j}\mu_{(m+k-j)}( \mu_{(j)}\otimes I)\\
&+\sum_{j>m+k}\frac{(-1)^{j+m+1}}{j\binom{j-1}{m}}\mu_{(m+k-j)}( \mu_{(j)}\otimes I) \N_{13}\,.
\end{align*}
Therefore, the degree $-2$ part of \eqref{borcherds2} for $n=0$ and $m,k\geq 0$ is equivalent to \eqref{e12}.
This completes the proof of \leref{lemproof4}.
\end{proof}

\begin{lemma}\label{lemproof5}
The $\pi_k$-projections of the Jacobi identity hold for $k=2,3,4$.
\end{lemma}
\begin{proof}
By Lemmas \ref{lemproof1}--\ref{lemproof3}, the $\pi_2$-projection of the Jacobi identity is equivalent to:
\begin{align*}
\sum_{i=1}^{L} &(-1)^{p_{i,a}} \Bigl(\frac{1}{\lambda+\partial}\phi_{i}a\Bigr)\psi_{i}([b_{\mu}c])
= \sum_{i=1}^{L} (-1)^{p_{i,a}} \Bigl(\frac{1}{\lambda+\partial }\phi_{i} a\Bigr) \bigl[ (\psi_{i} b)_{\mu}c \bigr] \\
&+\sum_{i=1}^{L} (-1)^{p_{i,a}+p_{i,b}} \Bigl(\frac{1}{\lambda+\partial}\phi_{i}a\Bigr) \bigl[b_{\mu}(\psi_{i}c) \bigr] \,.
\end{align*}
This equation follows from the fact that $\psi_{i}$ is a derivation of all products $\mu_{(n)}$ (see \prref{co3}(\ref{co3.4})),
and hence also of the bracket $[b_{\mu}c]$:
\[
\psi_{i}([b_{\mu}c]) = [(\psi_{i} b)_{\mu}c] + (-1)^{p_{i,b}} [b_{\mu}(\psi_{i}c)] \,.
\]

Similarly, the $\pi_3$-projection of the Jacobi identity is equivalent to:
\begin{align*}
\sum_{i=1}^{L} &(-1)^{p_{a,b}+p_{i,a}+p_{i,b}} \Bigl(\frac{1}{\mu+\partial}\phi_{i}b\Bigr) \bigl[a_{\lambda}(\psi_{i}c)\bigr] \\
= &-\sum_{i=1}^{L} (-1)^{p_{b,c}+p_{i,a}+p_{i,c}} \bigl[(\phi_{i} a)_{\lambda}c\bigr] \Bigl(\frac{1}{\mu+\partial }\psi_{i} b\Bigr) \\
&+\sum_{i=1}^{L} (-1)^{p_{a,b}+p_{i,b}} \Bigl(\frac{1}{\mu+\partial}\phi_{i}b\Bigr)\psi_{i}([a_{\lambda}c]) \,.
\end{align*}
Using commutativity of the product and the symmetry of $\N$ (\prref{co3}\eqref{co3.11}), we rewrite the second sum above as
\[
-\sum_{i=1}^{L} (-1)^{p_{a,b}+p_{i,b}} \Bigl(\frac{1}{\mu+\partial}\phi_{i}b\Bigr) \bigl[(\psi_{i} a)_{\lambda}c\bigr] \,.
\]
Hence, the $\pi_3$-projection of the Jacobi identity follows again from the fact that $\psi_{i}$ is a derivation.

Finally, the $\pi_4$-projection of the Jacobi identity is equivalent to:
\begin{align*}
\sum_{i=1}^{L} &(-1)^{p_{i,b}} \Bigl(\frac{1}{(\lambda+\mu)+\partial} \bigl[a_{\lambda}(\phi_{i}b)\bigr]\Bigr)(\psi_{i}c) \\
&= \sum_{i=1}^{L} (-1)^{p_{i,a}+p_{i,b}} \Bigl(\frac{1}{(\lambda+\mu)+\partial}\phi_{i}([a_{\lambda}b])\Bigr) (\psi_{i}c) \\
&-\sum_{i=1}^{L} (-1)^{p_{i,a}+p_{i,b}} \Bigl(\frac{1}{(\lambda+\mu)+\partial} \bigl[(\phi_{i}a)_{\lambda}b\bigr]\Bigr)(\psi_{i}c) \,,
\end{align*}
and it follows from the fact that $\phi_{i}$ is a derivation (\prref{co3}(\ref{co3.3})).
\end{proof}

\begin{lemma}\label{lemproof6}
The $\pi_k$-projections of the Jacobi identity hold for $k=5,6,7$.
\end{lemma}
\begin{proof}
By Lemmas \ref{lemproof1}--\ref{lemproof3}, the $\pi_5$-projection of the Jacobi identity is equivalent to:
\begin{align*}
\sum_{i,j=1}^{L} & (-1)^{p_{j,a}+p_{i,b}+p_{j,b}} \Bigl(\frac{1}{\lambda+\partial}\phi_{j}a\Bigr) \Bigl(\frac{1}{\mu+\partial}\phi_{i}b\Bigr) (\psi_{j}\psi_{i}c) \\
&= \sum_{i,j=1}^{L} (-1)^{p_{a,b}+p_{i,a}+p_{j,a}+p_{j,b}} \Bigl(\frac{1}{\mu+\partial}\phi_{j}b\Bigr) \Bigl(\frac{1}{\lambda+\partial}\phi_{i}a\Bigr) 
(\psi_{j}\psi_{i}c) \,,
\end{align*}
which holds due to the commutativity of the product and the property $[\psi_i,\psi_j]=0$ (see \prref{co3}\eqref{co3.2}).

Similarly, the $\pi_6$-projection of the Jacobi identity is equivalent to:
\begin{align*}
\sum_{i,j=1}^{L} &(-1)^{p_{j,a}+p_{i,b}} \Bigl(\frac{1}{(\lambda+\mu)+\partial} \Bigl(\Bigl(\frac{1}{\lambda+\partial}\phi_{j}a\Bigr)(\psi_{j}\phi_{i}b)\Bigr)\Bigr)(\psi_{i}c) \\
&= \sum_{i,j=1}^{L} (-1)^{p_{i,a}+p_{j,b}+p_{i,j}} \Bigl(\frac{1}{(\lambda+\mu)+\partial} \Bigl(\Bigl(\frac{1}{\lambda+\partial }\phi_{i} a\Bigr) (\phi_{j}\psi_{i}b)\Bigr)\Bigr) (\psi_{j}c) \,,
\end{align*}
and this follows from $[\phi_i,\psi_j]=0$ (\prref{co3}\eqref{co3.2}).

Finally, the $\pi_7$-projection of the Jacobi identity is equivalent to:
\begin{align*}
0 &=-\sum_{i,j=1}^{L} (-1)^{p_{i,a}+p_{j,a}+p_{j,b}} \Bigl(\frac{1}{(\lambda+\mu)+\partial}\Bigl((\phi_{j}\phi_{i}a) \Bigl(\frac{1}{\mu+\partial }\psi_{i} b\Bigr)\Bigr)\Bigr)(\psi_{j}c) \\
&+\sum_{i,j=1}^{L} (-1)^{p_{a,b}+p_{i,a}+p_{j,b}} \Bigl(\frac{1}{(\lambda+\mu)+\partial} \Bigl(\Bigl(\frac{1}{\mu+\partial}\phi_{j}b\Bigr)(\psi_{j}\phi_{i}a)\Bigr)\Bigr)(\psi_{i}c) \,.
\end{align*}
This follows from the commutativity of the product, the symmetry of $\N$, and $[\phi_i,\phi_j]=0$ (see \prref{co3}, parts \eqref{co3.11} and \eqref{co3.2}).
\end{proof}

This finishes the proof of \prref{pro3.11}. Taken together, Propositions \ref{pro3.5}--\ref{pro3.11} prove \thref{thm3.3}.

\section{Examples of Non-local Poisson Vertex Algebras}\label{sec4}

In this section, we present three examples of non-local Poisson vertex algebras.
The first two are obtained as associated graded of logarithmic vertex algebras from \cite{BV}. 
The third example is the potential Virasoro--Magri non-local PVA \cite{DK}, which we realize as the
associated graded of a new logVA.

\subsection{Associated graded of the potential free boson logVA}\label{ex4.1} 
In \cite[Section 4.2]{BV},  we constructed a logVA structure on the vector space 
\[
V=\mathbb{C}[x_{0},x_{1},x_2,\dots] \,,
\]
with a vacuum vector $\vac=1$, a translation operator $T\in \End(V)$ defined by 
\[
T1=0 \,, \qquad T(x_n)=(n+1)x_{n+1} \qquad (n\ge0)\,, 
\]
a braiding map 
\[
\N = -\partial_{x_0} \otimes \partial_{x_0} \,,
\]
and a state-field correspondence satisfying  
\begin{equation*}
Y(x_0,z) = \sum_{n=0}^\infty x_n z^n + \sum_{n=1}^\infty \partial_{x_n} \frac{z^{-n}}{-n}+ \partial_{x_0} \ze  \,.
\end{equation*}

The space of polynomials is graded by $\deg x_n = 1$ for $n\ge0$; the induced increasing filtration of $V$ 
is given by
\[
\F^{m}V= \Span\bigl\{x_{n_1}x_{n_2} \cdots x_{n_r} \,\big|\, r,n_i\in\ZZ_+ \,, \; r\le m\bigr\} \,, \qquad m\ge0
\]
(with the case $r=0$ corresponding to the empty product equal to $1$).
It is easy to check that this filtration satisfies the conditions of \deref{fil},
since by \cite[Proposition 2.17]{BDSK} it is enough to check them only for the generator $x_0$.
Therefore,
\[
\gr V \cong \mathbb{C}[x_{0},x_{1},x_2,\dots]
\]
is a non-local PVA with the usual commutative associative product,
and it is generated by $x_0$ as a differential algebra because $x_n=\partial^{(n)} x_0$.

To find the $\la$-bracket on $\gr V$, note that it is determined by $\{{x_{0}}_{\lambda}x_0\}$, since all other $\la$-brackets
can be obtained after applying the sesqui-linearity and Leibniz rule.
We have $\mu_{(n)}(x_0 \otimes x_0) = 0$ for $n\ge0$, as this is the coefficient of $z^{-n-1}$ in $Y(x_0,z)x_0$.
First, we calculate this $\la$-bracket in $V$, using \eqref{eq2.1} and \eqref{eq2.1b}:
\begin{align*}
\{{x_{0}}_{\lambda}x_0\} &= \mu_{(-1)}\Bigl(\N\, \Bigl(\frac{1}{\lambda+\partial }\,  x_0\otimes x_0\Bigr)\Bigr)
= -\mu_{(-1)}\Bigl(\frac{1}{\lambda+\partial }\,  1\otimes 1 \Bigr) \\
&= -\frac1\la \mu_{(-1)}(1 \otimes 1) = -\frac1\la 1\,.
\end{align*}
However, since $x_0\in\F^1 V$, its coset $x_0+\F^0 V$ is in $\gr^1 V$, and in $\gr V$ the induced $\la$-bracket becomes (cf.\ \eqref{e1.8}):
\[
\bigl\{ (x_{0}+\F^1 V)_{\lambda} (x_{0}+\F^1 V) \bigr\} = \{{x_{0}}_{\lambda}x_0\} + \F^0 V = 0 \in\gr^1 V \,.
\]
Therefore, the $\la$-bracket on the whole $\gr V$ is trivial.

The above situation is similar to what happens for ordinary vertex algebras generated by a Lie conformal algebra; in particular, for the free boson vertex algebra
(see e.g.\ \cite[Lemma 2.18]{BDSK}). To remedy it, as in \cite{BDSK}, we consider the central element $K$ of the free boson LCA (\exref{cur2}) 
to be in $\F^1 V$ instead of $\F^0 V$. We introduce the logVA
\[
V_K=\mathbb{C}[K,x_{0},x_{1},x_2,\dots] \,,
\]
with a vacuum vector $\vac=1$, a translation operator $T\in \End(V)$ defined by 
\[
T1=TK=0 \,, \qquad T(x_n)=(n+1)x_{n+1} \qquad (n\ge0)\,, 
\]
a braiding map 
\[
\N = -\frac{1}{2}K\partial_{x_0} \otimes \partial_{x_0}-\frac{1}{2}\partial_{x_0} \otimes K \partial_{x_0} \,,
\]
and a state-field correspondence such that $Y(K,z)=K$ and 
\begin{equation*}
Y(x_0,z) = \sum_{n=0}^\infty x_n z^n + \sum_{n=1}^\infty K\partial_{x_n} \frac{z^{-n}}{-n}+ K\partial_{x_0} \ze  \,.
\end{equation*}

We define an increasing filtration of $V_K$ by letting $\deg K=\deg x_n = 1$ for all $n\ge0$, i.e.,
\[
\F^{m} V_K = \Span\bigl\{K^s x_{n_1} \cdots x_{n_r} \,\big|\, s,r,n_i\in\ZZ_+ \,, \; s+r\le m\bigr\} \,, \qquad m\ge0 \,.
\]
Then
\[
\gr V_K \cong \mathbb{C}[K,x_{0},x_{1},x_2,\dots]
\]
is graded by $\deg K=\deg x_n = 1$. The filtration of $V_K$ satisfies the conditions of \deref{fil};
hence $\gr V_K$ is a non-local PVA. As above, the product in $\gr V_K$ is the usual product in the polynomial algebra,
and $\gr V_K$ is generated by $K$ and $x_0$ as a differential algebra, where $\partial K=0$ and $x_n=\partial^{(n)} x_0$.

We have in $V_K$:
\[
\{{x_{0}}_{\lambda}x_0\} = -\frac1\la K \,,
\]
which induces the unique $\la$-bracket in $\gr V_K$ determined by (cf.\ \eqref{e1.8}):
\[
\bigl\{ (x_{0}+\F^1 V_K)_{\lambda} (x_{0}+\F^1 V_K) \bigr\} = \{{x_{0}}_{\lambda}x_0\} + \F^0 V_K = -\frac1\la (K + \F^0 V_K) \,.
\]
Moreover, the element $K + \F^0 V_K \in\gr^1 V_K$ is central, i.e., it has a trivial $\la$-bracket with any element of $\gr V_K$.
The quotient
\[
\V := \gr V_K / (K-1)\gr V_K \cong \mathbb{C}[x_{0},x_{1},x_2,\dots]
\]
is a non-local PVA isomorphic to the potential free boson non-local PVA defined in Example \ref{ex2.11}.
Notice that while $\gr V_K$ is graded, $\V$ is not graded as a non-local PVA, because $\deg K=1$ but $\deg 1=0$.

\subsection{Associated graded of the Gurarie--Ludwig logVA}\label{ex4.3} 

In this subsection, we consider the Gurarie--Ludwig's logVA constructed in \cite[Section 4.4]{BV},
which was motivated by an example of logarithmic conformal field theory from \cite{GL1,GL2,G2}.

This logVA $V$ is a vector superspace linearly spanned by the vacuum vector $\vac$ and monomials of the form
\begin{equation*}
a^{1}_{n_1} \cdots a^{r}_{n_r} \vac \quad\text{where}\quad
n_i \le -2 \,, \; a^i\in\{L,\ell,\xi,\bar{\xi}\} \,, \; 1\le i\le r \,.
\end{equation*}
The parity is determined by letting all $L_n$, $\ell_n$ be even and all $\xi_n$, $\bar\xi_n$ be odd.
In order to recall the logVA structure on $V$, we define odd linear operators $\eta,\bar{\eta}\in \End(V)$ by $\eta\vac=\bar{\eta}\vac=0$ and
\[[\eta, a_{n}]=(\eta a)_{n}\,, \quad [\bar{\eta}, a_{n}]=(\bar{\eta} a)_{n}\,,  \qquad a\in \{L,\ell,\xi,\bar{\xi}\} , \;\;
n\le -2 \,,\]
where
\[
\eta\ell = -2\xi \,, \qquad \bar\eta\ell = 2\bar\xi \,, \qquad \eta\bar\xi = \bar\eta\xi = L \,,
\]
and
\[
\eta^2=\bar\eta^2=0 \,, \qquad \eta\bar\eta=-\bar\eta\eta \,.
\]
Then $\eta$ and $\bar\eta$ are odd derivations of $V$. 
The superspace $V$ is a module over the associative superalgebra $\A$ with generators
\[
\bigl\{ L_{n},\,\ell_{n},\, \xi_{n},\, \bar{\xi}_{n},\, \eta,\, \bar{\eta} \,\big|\, n\in \ZZ \bigr\} \,,
\]
subject to the relations listed in \cite[Lemma 4.5]{BV}.

The translation operator on $V$ is $T=L_{-1}$. More explicitly, $T$ can be determined by $T\vac=0$ and the commutators
\begin{align*}
[T,L_n] &= -(n+1)L_{n-1} \,, & &[T,\ell_{n}]=-(n+1)\ell_{n-1} \!-\! \xi_{n-1}\bar{\eta} \!-\! \bar{\xi}_{n-1}\eta \,, \\
[T, {\xi}_{n}]&=-(n+1){\xi}_{n-1}+\frac{1}{2}L_{n-1}\eta\,, &
&[T, \bar{\xi}_{n}]=-(n+1)\bar{\xi}_{n-1}-\frac{1}{2}L_{n-1}\bar{\eta}\,,
\end{align*}
which follow from the relations in $\A$.
The braiding map on $V$ is defined by
\begin{equation}\label{equ2.14s}
\N = \frac{1}{2}(\bar{\eta} \otimes \eta -  \eta \otimes \bar\eta) \,.
\end{equation}
Explicitly, we have:
\begin{align*}
\N(\bar\xi\otimes\xi) &= - \N(\xi\otimes\bar\xi) = \frac12 L \otimes L \,, &
\N(\ell\otimes\ell) &= 2\xi\otimes\bar\xi-2\bar\xi\otimes\xi \,, \\
\N(\ell\otimes\xi) &= \xi\otimes L \,, & \N(\xi\otimes\ell) &= L\otimes\xi \,, \\
\N(\ell\otimes\bar\xi) &= \bar\xi\otimes L \,, & \N(\bar\xi\otimes\ell) &=  L\otimes\bar\xi \,, 
\end{align*}
and $\N=0$ 
on all other tensors not listed,
where we identify $a$ with $a_{-2}\vac\in V$ for $a\in\{L,\ell,\xi,\bar\xi\}$.

The modes of the generators are given by
\begin{equation*}
a_{(n+\N)} = a_{n-1} \quad\text{for}\quad n\in\ZZ \,, \;\; a\in\{L,\ell,\xi,\bar\xi\} \, .
\end{equation*}
The products $\mu_{(n)}(a \otimes b) = a_{(n+\N)} b$ for $n\ge0$ and ordered pairs $a,b\in\{L,\ell,\xi,\bar\xi\}$ follow from the relations in $\A$,
and are given as follows \cite{BV}:
\begin{equation}\label{equ2.14a}
\begin{split}
L_{0}b &=2b\, , \quad L_{n}b=0\,  , \qquad n>0 \,, \; b\in\{L,\xi,\bar\xi\} \,. \\
L_{0}\ell &=2\ell+L \, , \quad L_{2}\ell=\be\vac \,, \quad L_1\ell=L_{n}\ell=0\,  , \qquad n>2 \,, \\
\xi_2\bar\xi &= \frac{\be}2\vac \,, \quad \xi_1\bar\xi = 0 \,, \quad \xi_0\bar\xi = \frac14(4\ell+L) \,, \quad \xi_{-1} \bar\xi = \frac{T}8(4\ell+L) \,, \\
\ell_0\xi &= \frac12 \xi \,, \quad \ell_{-1}\xi = \frac34 T\xi \,, \quad \ell_0\bar\xi =\frac12 \bar\xi \,, \quad \ell_{-1}\bar\xi = \frac34 T\bar\xi 
\,, \\
\ell_0\ell &=  \ell \,, \quad \ell_{-1}\ell = \frac12 T\ell \,, \\
\xi_m\bar\xi &= 0 \,, \quad \ell_n b = 0 \,, \qquad m>2 \,, \; n>0 \,, \; b\in\{\ell,\xi,\bar\xi\} \,,
\end{split}
\end{equation}
where $\be\in\CC$ is a parameter.
From these formulas for $\mu_{(n)}(a \otimes b)$, the above action of $\N$, and the definition \eqref{eq2.1}, \eqref{eq2.1b},
we find on $V$ the following $\lambda$-brackets (where $\partial=T$):
\begin{equation}\label{equ2.14b}
\begin{split}
\{L_{\lambda}L\}&=(2\lambda+\partial)L\, , \quad 
\{L_{\lambda}\xi\}=(2\lambda+\partial)\xi\, , \quad
\{L_{\lambda}\bar{\xi}\}=(2\lambda+\partial)\bar{\xi}\, , \\
\{L_{\lambda}\ell\} &=(2\lambda+\partial)\ell+\la L + \frac{\lambda^3}{6} \be\vac \, , \\
\{\xi_{\lambda}\bar{\xi}\}&=\frac{1}{8}(2\lambda+\partial)(4\ell+L) + \frac{\lambda^3}{12} \be\vac - \frac{1}{2}\Bigl(\frac{1}{\lambda+\partial}L\Bigr)L \, , \\
\{\ell_{\lambda}\xi\}&= \frac14(2\lambda+3\partial)\xi + \Bigl(\frac{1}{\lambda+\partial}\xi\Bigr)L\, , \\
\{\ell_{\lambda}\bar{\xi}\} &= \frac14(2\lambda+3\partial)\bar\xi +\Bigl(\frac{1}{\lambda+\partial}\bar{\xi}\Bigr)L\, , \\
\{\ell_{\lambda}\ell\}&=\frac12(2\lambda+\partial)\ell+2\Bigl(\frac{1}{\lambda+\partial}\xi\Bigr)\bar{\xi}-2\Bigl(\frac{1}{\lambda+\partial}\bar{\xi}\Bigr){\xi}\,.
\end{split}
\end{equation}

In order to introduce a filtration on $V$, we define the weighted degree of a monomial as
\[
\deg (a^{1}_{n_1} \cdots a^{r}_{n_r}\vac)=r+s \,,
\]
where $s$ is the number of factors in the product with $a^i\in\{\ell,\bar{\xi}\}$. In other words, we let $\deg\vac=0$ and
\[
\deg L_n = \deg \xi_n = 1 \,, \quad \deg \ell_n = \deg \bar\xi_n = 2 \,, \qquad n\le-2.
\]
Then we define the filtration 
\[
\F^{m}V = \Span\bigl\{ v=a^{1}_{n_1} \cdots a^{r}_{n_r}\vac \,\big|\, \deg v\leq m \bigr\}, \qquad m\ge0\,.
\]
In particular, we have
\[
\F^0 V = \CC\vac \,, \qquad \F^1 V = \Span\{\vac,L,\xi\} \,, \qquad \ell,\bar\xi\in\F^2 V \,.
\]
It is easy to check that this filtration satisfies the conditions of \deref{fil},
since by \cite[Proposition 2.17]{BDSK} it is enough to check them only for the generators $L,\ell,\xi,\bar\xi$.

Therefore, $\gr V$ is a non-local PVA, generated as a differential algebra by the images of $L,\ell,\xi,\bar\xi$,
which we denote by the same letters as the corresponding elements of $V$.
One can show that $V$ has a Poincar\'e--Birkhoff--Witt basis consisting of ordered monomials; hence
\begin{equation*}
\gr V \cong \CC\bigl[\partial^i L, \partial^i \ell, \partial^i \xi, \partial^i \bar{\xi}\,\bigr]_{i\in\ZZ_+}
\end{equation*}
is an algebra of differential polynomials in the generators. The grading in $\gr V$ is given by:
\[
\deg 1 = \deg\partial = 0 \,, \qquad \deg L = \deg \xi = 1 \,, \qquad \deg \ell = \deg \bar\xi = 2 \,.
\]
The $\lambda$-brackets on $\gr V$ are induced from those on $V$ given by \eqref{equ2.14b},
by keeping only terms of the correct degree so that $\{\cdot_\la\cdot\}$ has degree $-1$ as in \eqref{eq2.1a}.
In this way, we obtain on $\gr V$ the $\lambda$-brackets \eqref{equ2.14} with $K=0$.

If we want to retain the parameter $\be\in\CC$, we need to modify slightly the logVA $V$, 
similarly to the case of the free boson logVA from Section \ref{ex4.1}.
The new logVA will be $V_{K}=\CC[K]\otimes V$, where $K$ is an even element such that $T(K)=0$ and $Y(K,z)=K$.
The braiding map $\N$ remains as before \eqref{equ2.14s}, and we let $\eta(K)=\bar\eta(K)=0$.
The only change in the products \eqref{equ2.14a} is that we replace $\be\vac$ with $K$ in the right-hand side.
We define a grading on $\CC[K]$ by $\deg K=2$; this induces a filtration on $\CC[K]$ and on the tensor product $V_K$.
One checks again that the conditions of \deref{fil} hold. Therefore, we have the non-local PVA
\begin{equation*}
\gr V_K \cong \CC\bigl[K,\partial^i L, \partial^i \ell, \partial^i \xi, \partial^i \bar{\xi}\,\bigr]_{i\in\ZZ_+} \,,
\end{equation*}
in which $\partial K=0$ and $\deg K=2$.
Now the induced $\lambda$-brackets on $\gr V_K$ are given exactly by \eqref{equ2.14}.
Then, for $\be\in\CC$, the quotient
\[
\V_\be := \gr V_K /(K-\be 1)\gr V_K \cong \CC\bigl[\partial^i L, \partial^i \ell, \partial^i \xi, \partial^i \bar{\xi}\,\bigr]_{i\in\ZZ_+}
\]
is the non-local PVA from Example \ref{ex2.13}.
Notice that while $\gr V_K$ is graded, $\V_\be$ is not graded as a non-local PVA, because $\deg K=2$ but $\deg 1=0$.

\subsection{Potential Virasoro--Magri non-local PVA}\label{sec4.3}
In this subsection, we consider the potential Virasoro--Magri non-local PVA from Example \ref{ex2.12} and \cite{DK}.
Recall that, for any central charge $c\in\CC$, this non-local PVA
\[\V_c=\C[u,u',u'', \dots] \]
is an algebra of differential polynomials in the generator $u$, where $u'=\partial u$, $u''=\partial^2 u$, etc.
The $\lambda$-bracket in $\V_c$ is uniquely determined by $\{u_{\lambda}u\}$, which is given by (cf.\ \eqref{equ2.12}):
\begin{equation}\label{eq5.0}
\{u_{\lambda}u\}=-\frac{1}{\lambda}u'-\frac{1}{\lambda+\partial}u'-\frac{\lambda}{12}c\, .
\end{equation}
In particular, by sesqui-linearity we have
\begin{equation}\label{eq5.0b}
\{u'_{\lambda}u'\}=(2\lambda+\partial)u'+\frac{\lambda^3}{12}c\,,
\end{equation}
which is the $\lambda$-bracket of the Virasoro LCA with central charge $c$ (see e.g.\ \cite{K1}).
We will build an example of a filtered logVA $V$ such that $\gr V \cong\V_0$.
Here we only provide an outline of the construction and omit some of the details.

We start by assuming the existence of a logVA $V$, generated by an element $u\in V$ with a $\lambda$-bracket \eqref{eq5.0},
where $\partial$ is replaced by the translation operator $T$.
Note that \eqref{eq5.0} follows from \eqref{eq2.1}, \eqref{eq2.1b} if in $V$ we have that
\begin{equation}\label{eq5.2}
\mu_{(n)}(u\otimes u) = -\delta_{n,1}\frac{c}{12} \vac\,, \qquad n\geq 0 \,,
\end{equation}
and
\begin{equation*}
\S(u\otimes u) = -\vac\otimes u' - u' \otimes\vac \,,
\end{equation*}
where $u'=Tu$. 
We will suppose that $V$ has a derivation $D=\partial_{u}$ such that $Du=\vac$.
Since every derivation 
commutes with $T$, this implies that $Du'=Du''=\cdots=0$.
We will define the braiding map on $V$ by
\begin{equation}\label{eq5.1}
\S:=-D\otimes T-T\otimes D\,. 
\end{equation}
Then the map $\N$ is locally nilpotent on $V\otimes V$, because $D$ is locally nilpotent on $V$.

We can derive the relations satisfied by the modes of $u$ from 
the Borcherds identity \eqref{borcherds2} with $n=0$ applied to $u \otimes u \otimes v$ for arbitrary $v\in V$,
by using \eqref{eq5.2} and \eqref{eq5.1}. We obtain:
\begin{equation}\label{eq5.3}
\begin{split}
\bigl[u_{(m+\S)}, u_{(k+\S)}\bigr] &= \Bigl(\frac{1-\delta_{m,-1}}{m+1}-\frac{1-\delta_{k,-1}}{k+1}\Bigr)u'_{(m+k+1+\S)}\\
&- \delta_{m+k\ge0} \, (m-k)(m+k+1) \frac{c}{24}D^{m+k} \,,
\end{split}
\end{equation}
for all $m,k\in\ZZ$, where
\[
\delta_{n\ge0} = \begin{cases}
1, \;\;\text{ if }\; n\ge 0\,, \\ 0, \;\;\text{ if }\;  n<0\,.
\end{cases}
\]
The modes of $u'$ can be expressed in terms of the modes of $u$ as follows. First, by translation covariance
(see \prref{co}\eqref{cor2.2}, \eqref{cor2.3}), we have for any $v\in V$ and $n\in\ZZ$:
\begin{equation}\label{eq5.6}
\begin{split}
u'_{(n+\S)}v &= [T, u_{(n+\S)}]v
= -n \,\mu_{(n-1)}(u\otimes v) - \mu_{(n-1)}\S(u\otimes v) \\
&=-n \, u_{(n-1+\S)}v+\delta_{n,0} \, Tv+u'_{(n-1+\S)}(Dv) \,.
\end{split}
\end{equation}
Applying this formula 
$N$ times, where $D^Nv=0$, we obtain
\begin{equation}\label{eq5.5}
u'_{(n+\S)}v 
= \delta_{n\geq 0} \, TD^{n}v - \sum_{j=0}^\infty (n-j)u_{(n-1-j+\S)}(D^{j}v) \,.
\end{equation}
Note that the sum over $j$ runs only up to $N-1$, but we wrote it in a form independent of $N$.

The construction of the logVA $V$ now proceeds similarly to the construction of the Gurarie--Ludwig's logVA from \cite[Section 4.4]{BV};
see also \seref{ex4.3} above. We introduce the unital associative algebra $\A$ with generators
\[
\bigl\{ u_{n},\,D,\,T \,\big|\, n\in \ZZ \bigr\} \,,
\]
subject to the relations $(n,m,k\in\ZZ)$:
\begin{equation}\label{eq5.11}
\begin{split}
[D,T] &= 0 \,, \\
[D,u_{n}] &= \delta_{n,-1} 1  \,, \\
[T,u_{n}] &= \delta_{n\geq 0} \, TD^{n} - \sum_{j\geq 0}(n-j)u_{n-1-j} D^{j} \,, \\
[u_{m}, u_{k}] &= \Bigl(\frac{1-\delta_{m,-1}}{m+1}-\frac{1-\delta_{k,-1}}{k+1}\Bigr) [T,u_{m+k+1}] \\
&- \delta_{m+k\ge0} \, (m-k)(m+k+1) \frac{c}{24}D^{m+k} \,,
\end{split}
\end{equation}
where $c\in\CC$ is a parameter.
The algebra $\A$ is $\ZZ$-graded by setting
\[\deg 1=0 \,, \qquad \deg u_{n}=-n \,, \qquad \deg D=-1 \,, \qquad \deg T=1\,.\]
The infinite sums in the relations \eqref{eq5.11} are convergent in the topology of $\A$, in which $\lim_{j\to+\infty} a_j b_j = 0$ for any sequence
$a_j,b_j\in\A$ with $\deg a_j = d+j$ and $\deg b_j=-j$ for some fixed $d\in\ZZ$.

The algebra $\A$ admits a module $V$, generated by an element $\vac\in V$ satisfying the relations
\[u_{n}\vac=D\vac=T\vac=0 \,, \qquad n\geq 0 \,, \]
and all other relations in $V$ follow from these and the relations in $\A$.
The module $V$ is linearly spanned by monomials of the form 
\[u_{-n_1} \cdots u_{-n_r} \vac\,  \qquad (r\in\ZZ_+, \, n_{i}\ge1)\, . \]
The subset of ordered monomials, such that $n_1\ge\cdots\ge n_r\ge1$, is a basis for $V$.
Moreover, $V$ has a $\ZZ_+$-grading compatible with the grading of $\A$, such that $\deg\vac=0$.
In particular, for every homogeneous $v\in V$, we have
\[
u_{n}v=D^nv=0 \,, \qquad n>\deg v \,.
\]

We endow $V$ with the following structure of a logVA. The vacuum vector is $\vac$, the translation operator is $T$, the braiding map $\N$ is given by \eqref{eq5.1}, and the modes of $u:=u_{-1}\vac \in V$ are given by $u_{(n+\N)}=u_n$ for $n\in\ZZ$.
Then the logarithmic field $Y(u,z)$ is determined from \eqref{p-modes1}; explicitly, one finds that
\[
Y(u,z) = \ze T + \sum_{k\in\ZZ_+} \sum_{n\in\ZZ} \frac{1}{k!}\ze^{k}z^{-n-1} \, (\ad T)^{k}(u_n) \, D^{k} \,.
\]
The relations \eqref{eq5.11} imply \eqref{eq5.6} and \eqref{eq5.5}, from which one can derive the translation covariance of the field $Y(u,z)$.
Similarly, \eqref{eq5.3} is equivalent to the Borcherds identity \eqref{borcherds2} with $n=0$ applied to $u\otimes u\otimes v$ for arbitrary $v\in V$,
from which one can deduce the locality of $Y(u,z)$ with itself. There is a subtlety that in the locality condition \eqref{eq2.12}
for $a=b=u$ and $c=v$, the power $N$ will depend on the vector $v$. This is due to the fact that, while $\N$ is locally nilpotent,
not all of its components $\phi_i$ are (namely, $T$ is not locally nilpotent); see \cite[Remark 3.27]{BV}.
Nevertheless, we can apply the Existence Theorem \cite[Theorem 3.2]{BV} to conclude that $V$ is a (generalized) logVA
with a generating field $Y(u,z)$.
This and further generalizations of the notion of a logarithmic vertex algebra will be explored in a future work.

We introduce an increasing filtration of $V$:
\[
\F^{m} V = \Span\bigl\{u_{-n_1} \cdots u_{-n_r}\vac \,\big|\, 0 \le r\le m \,, \; 1\le n_{i} \bigr\} \,, \qquad m\ge0 \,,
\]
which satisfies the conditions of \deref{fil}. Hence, $\gr V$ inherits an induced structure of a non-local PVA. 
It is easy to see that $\gr V \cong \V_0$ is the potential Virasoro--Magri non-local PVA with central charge $c=0$.
In order to obtain any $c\in\CC$, we modify the definition of the logVA $V$ as in Sections \ref{ex4.1} and \ref{ex4.3}.
Namely, we consider the logVA $V_C=\CC[C]\otimes V$, in which $C$ is a central element with 
\[D(C)=T(C)=0 \,, \qquad \deg C=1 \,, \]
and $c\vac$ is replaced with $C$ in the products \eqref{eq5.2}. Then we have
\[
\V_c \cong \gr V_C /(C-c 1)\gr V_C \,.
\]

Finally, we will derive the relations satisfied by the modes of $L:=u'=Tu=u_{-2}\vac$ in the logVA $V$.
Observe that
$\N(L \otimes L) = 0$
and
\[
\mu_{(n)}(L \otimes L) = L_{(n+\S)}L = \begin{cases}
T(L) \,, \;\; & n=0 \,, \\
2L \,, & n=1 \,, \\
c/2 \,, & n=3 \,, \\
0 \,, & n=2 \;\text{ or }\; n>3 \,,
\end{cases}
\]
which means that the $\lambda$-bracket $\{L_\lambda L\}$ is given by \eqref{eq5.0b}, 
as in the Virasoro LCA with central charge $c$.
However, there exist vectors $v$ such that $\N(L \otimes v) \ne 0$; for example,
$\N(L \otimes u) = -T(L) \otimes \vac$.
Hence, the modes $L_n := L_{(n+1+\S)}$ do not satisfy the commutation relations of the Virasoro Lie algebra.
An explicit calculation, using the Borcherds identity \eqref{borcherds2} with $n=0$ applied to $L \otimes L \otimes v$
for arbitrary $v\in V$, gives the relations ($m,k\in\ZZ$):
\begin{equation}\label{eq5.9}
\begin{split}
[&L_{m}, L_{k}] = (m-k)\sum_{j= 0}^\infty (j+1)L_{m+k-j}D^{j}\\
&+\delta_{m+k\ge0} \, (m - k)\bigl((m-k)^2 - m - k - 4\bigr) \binom{m+k + 3}{3} \frac{c}{96}D^{m+k}\, .
\end{split}
\end{equation}
In particular, these reduce to the Virasoro commutation relations when applied to the kernel of $D$.
Moreover, since $D$ is a derivation of $V$, we have:
\begin{equation}\label{eq5.10}
[D,L_n] = (DL)_{(n+1+\S)} = 0 \,, \qquad n\in\ZZ \,.
\end{equation}

\begin{remark}
The non-linear Lie algebra with relations \eqref{eq5.9}, \eqref{eq5.10} for $c=0$
can be realized in terms of vector fields on the circle as follows:
\[L_{n}=-\frac{t^{3+n}}{(t-D)^{2}} \, \frac{d}{dt} \,, \qquad n\in\ZZ\, , \]
where $D$ is central and we use the expansion
\[\frac{1}{(t-D)^{2}}=\sum_{j= 0}^\infty(j+1)t^{2-j}D^{j} \,.\]
Then the usual Gelfand--Fuchs cocycle on vector fields produces a cocycle equivalent to that in \eqref{eq5.9}.
\end{remark}

\section*{Declarations}

\subsection*{Funding}
The first author was supported in part by a Simons Foundation grant 584741. The second author was supported in part by UK Research and Innovation grant MR/S032657/1. 

\subsection*{Competing interests}
The authors have no competing interests to declare that are relevant to the content of this article.

\subsection*{Data availability statement}
Data sharing not applicable to this article as no datasets were generated or analysed during the current study.

\subsection*{Acknowledgments}
The second author thanks Marco Aldi for stimulating discussions on a related subject. He is grateful to Tomoyuki Arakawa and the organizers of the conference Vertex Algebras and Representation Theory at CIRM in June 2022, where some of the results of this work were announced. 
Both authors are grateful to Nikolay M. Nikolov for interesting discussions, and to the Institute for Nuclear Research and Nuclear Energy of the Bulgarian Academy of Sciences for the hospitality during June 2022.

\bibliographystyle{amsalpha}

\begin{thebibliography}{DSKVW}
\bibitem[B]{B}
Bakalov, B.:
\textit{Twisted logarithmic modules of vertex algebras}.
Commun. Math. Phys. \textbf{345}, 355--383 (2016)

\bibitem[BDSK]{BDSK}
Bakalov, B., De Sole, A., Kac, V.G.:
\textit{Computation of cohomology of vertex algebras}. 
Jpn. J. Math. \textbf{16}, 81--154 (2021)

\bibitem[BK]{BK}
Bakalov, B., Kac, V.G.:
\textit{Field algebras}.
Internat. Math. Res. Not.
\textbf{2003}, no. 3, 123--159 (2003)

\bibitem[BS1]{BS1}
Bakalov, B., Sullivan, M.:
\textit{Twisted logarithmic modules of free field algebras}.
J. Math. Phys. \textbf{57}, 061701, 18 pp. (2016)

\bibitem[BS2]{BS2}
Bakalov, B., Sullivan, M.:
\textit{Twisted logarithmic modules of lattice vertex algebras}.
Trans. Amer. Math. Soc. \textbf{371}, 7995--8027 (2019)

\bibitem[BV]{BV}
Bakalov, B.N., Villarreal, J.J.:
\textit{Logarithmic vertex algebras}.
Transf. Groups, published online, 63 pp. (2022)


\bibitem[Bo]{Bo}
Borcherds, R.E.:
\textit{Vertex algebras, Kac--Moody algebras, and the Monster}.
Proc. Nat. Acad. Sci. USA \textbf{83}, 3068--3071 (1986)

\bibitem[CR]{CR}
Creutzig T., Ridout, D.:
\textit{Logarithmic conformal field theory: beyond an introduction}. 
J. Phys. A \textbf{46}, 494006, 72 pp. (2013)

\bibitem[DSK]{DK}
De Sole, A., Kac, V.G.:
 \textit{Non-local Poisson structures and applications to the theory of integrable systems}.
{Japanese J. Math.} \textbf{8}, 233--347 (2013)

\bibitem[DSKV]{DKV} 
De Sole A., Kac V.G., Valeri D.:
\textit{Dirac reduction for Poisson vertex algebras}. 
{Comm. Math. Phys.} \textbf{ 331}, 1155-1190 (2014)

\bibitem[DSKV2]{DKV2} 
De Sole A., Kac V.G., Valeri D.:
\textit{Integrability of Dirac reduced bi-Hamiltonian equations}.
{Trends in Contemporary Math., Springer INDAM} \textbf{ 8}, 13-32 (2014).

\bibitem[DSKVW]{DKVW}
De Sole A., Kac V.G., Valeri D., Wakimoto M.:
\textit{Local and non-local multiplicative Poisson vertex algebras and difference equations}.
{Comm. Math. Phys.} \textbf{ 370},  1019–1068 (2019)



\bibitem[DMS]{DMS}
Di Francesco, P., Mathieu, P., S\'en\'echal, D.:
\textit{Conformal field theory}.
Graduate Texts in Contemporary Physics, Springer--Verlag, New York, 1997


\bibitem[FB]{FB}
Frenkel, E., Ben-Zvi, D.:
\textit{Vertex algebras and algebraic curves}.
Math. Surveys and Monographs, 88,
Amer. Math. Soc., Providence, RI, 2001; 2nd ed., 2004


\bibitem[FLM]{FLM}
Frenkel, I.B., Lepowsky, J.,  Meurman, A.:
\textit{Vertex operator algebras and the Monster}.
Pure and Appl. Math., 134,  Academic Press,  Boston, 1988



\bibitem[G1]{G1}
Gurarie, V.:
 \textit{Logarithmic operators in conformal field theory}.
 Nuclear Phys. B \textbf{410}, 535--549 (1993)

\bibitem[G2]{G2}
Gurarie, V.:
  \textit{Logarithmic operators and logarithmic conformal field theories}.
  J. Phys. A: Math. Theor. \textbf{46}, 494003, 18pp. (2013)

\bibitem[GL1]{GL1}
Gurarie, V., Ludwig, A.W.W.:
\textit{Conformal algebras of two-dimensional disordered systems}. 
J. Phys. A \textbf{35}, no. 27, L377–L384 (2002)

\bibitem[GL2]{GL2}
Gurarie, V., Ludwig, A.W.W.:
\textit{Conformal field theory at central charge $c=0$ and two-dimensional critical systems with quenched disorder}. 
In: From fields to strings: circumnavigating theoretical physics, vol. 2, 1384--1440, World Sci. Publ., Singapore, 2005;
hep-th/0409105

\bibitem[H]{H}
Huang, Y.-Z.:
\textit{Generalized twisted modules associated to general automorphisms of a vertex operator algebra}. 
Commun. Math. Phys. \textbf{298}, 265--292 (2010)

\bibitem[K1]{K1}
Kac, V.G.:
\textit{Vertex algebras for beginners}. 
University Lecture Series, 10, 
Amer. Math. Soc., Providence, RI, 1996; 2nd ed., 1998

\bibitem[K2]{K2}
Kac, V.G.:
\textit{Introduction to vertex algebras, Poisson vertex algebras, and integrable Hamiltonian PDE}.
{Adv. Math.} \textbf{281}, 1025--1099, (2015).

\bibitem[KRR]{KRR}
Kac, V.G., Raina, A.K., Rozhkovskaya, N.:
\textit{Bombay lectures on highest weight representations of infinite dimensional Lie algebras}. 
2nd ed., Advanced Ser. in Math. Phys., 29. 
World Sci. Pub. Co. Pte. Ltd., Hackensack, NJ, 2013



\bibitem[LL]{LL}
Lepowsky, J., Li, H.:
\textit{Introduction to vertex operator algebras and their representations}.
Progress in Math., 227, Birkh\"auser Boston, Boston, MA, 2004

\bibitem[Li]{Li}
Li, H.:
\textit{Vertex algebras and vertex Poisson algebras}. 
Commun. Contemp. Math. Phys. \textbf{6}, 61--110 (2004)


\end{thebibliography}

\end{document}